\documentclass[fontsize=12pt, a4paper, bibliography = totocnumbered,
abstract=true, oneside]{scrartcl}
\usepackage[utf8]{inputenc}
\usepackage[T1]{fontenc}
\usepackage{amsthm} 
\usepackage{thmtools}
\usepackage{amssymb} 
\usepackage{mathtools} 
\usepackage{microtype} 

\usepackage{tikz} 
\usetikzlibrary{shapes.geometric}
\usepackage{tikz-cd} 
\usepackage{quiver}
\usepackage{spectralsequences} 
\usepackage{float} 
\usepackage{scalerel}
\usepackage{eucal}
\usepackage{adjustbox}
\usepackage[backend = biber, style = alphabetic, urldate = long]{biblatex}
\bibliography{References.bib}
\usepackage{hyperref}
\usepackage[noabbrev]{cleveref}

\theoremstyle{plain}
\newtheorem{theorem}{Theorem}[section]
\newtheorem{proposition}[theorem]{Proposition}
\newtheorem{corollary}[theorem]{Corollary}
\newtheorem{lemma}[theorem]{Lemma}

\newtheorem{thmx}{Theorem}

\theoremstyle{definition}
\newtheorem{example}[theorem]{Example}
\newtheorem{definition}[theorem]{Definition}
\newtheorem{remark}[theorem]{Remark}
\newtheorem{construction}[theorem]{Construction}
\newtheorem{notation}[theorem]{Notation}

\DeclareMathOperator{\alg}{Alg}
\DeclareMathOperator{\calg}{CAlg}

\DeclareMathOperator{\rmod}{RMod}

\newcommand{\Segc}{\mathrm{Seg}^{\mathrm{cpl}}(\Omega^{c})}
\newcommand{\Segck}[1][k]{\mathrm{Seg}^{\mathrm{cpl}}(\Omega_{\leq #1}^{c})}
\newcommand{\ei}{\mathbb{E}_{\infty}}
\newcommand{\enn}{\mathbb{E}_{n}}
\newcommand{\eone}{\mathbb{E}_{1}}
\newcommand{\enul}{\mathbb{E}_{0}}
\newcommand{\ai}{\mathbb{A}_{\infty}}

\newcommand{\an}[1][n]{\mathbb{A}_{#1}}

\DeclareMathOperator{\fun}{Fun}
\DeclareMathOperator{\lfun}{LFun}
\DeclareMathOperator{\rfun}{RFun}
\DeclareMathOperator{\map}{Map}

\DeclareMathOperator{\mul}{Mul}
\DeclareMathOperator{\Nu}{nu}

\DeclareMathOperator{\bound}{bound}
\DeclareMathOperator{\cobound}{cobound}

\DeclareMathOperator{\red}{red}
\DeclareMathOperator{\ev}{ev}
\DeclareMathOperator{\cob}{cob}

\DeclareMathOperator{\id}{id}

\DeclareMathOperator{\op}{op}
\DeclareMathOperator{\gr}{Gr}
\DeclareMathOperator{\gd}{gd}

\DeclareMathOperator{\cofib}{cofib}
\DeclareMathOperator{\fib}{fib}
\DeclareMathOperator{\colim}{colim}

\DeclareMathOperator{\pic}{Pic}

\DeclareMathOperator{\sym}{Sym}
\DeclareMathOperator{\Th}{Th}
\DeclareMathOperator{\End}{End}

\newcommand{\s}{\mathbb{S}}

\newcommand{\E}{\mathbb{E}}
\newcommand{\Z}{\mathbb{Z}}

\newcommand{\fourmodule}{_{\varphi}1_{\mathcal{C}} }
\newcommand{\sfourmodule}{_{\tilde{4}}\nu\mathbb{S}}

\DeclareMathOperator{\Op}{Op}

\newcommand{\redOp}{\mathrm{Op}^{\mathrm{red}}}
\newcommand{\unOp}{\mathrm{Op}^{\mathrm{un}}}
\newcommand{\eOp}[1][n]{(\mathrm{Op})_{\mathbb{E}_{0}/}}

\newcommand{\prl}{\mathcal{P}\mathrm{r}^{\mathrm{L}}}
\newcommand{\prst}{\mathcal{P}\mathrm{r}^{\mathrm{St}}}
\newcommand{\prgd}{\mathcal{P}\mathrm{r}^{\gd}_{0 \leq}}
\newcommand{\prp}{\mathcal{P}\mathrm{r}^{\gd}_{(p),0 \leq}}
\newcommand{\prtwo}{\mathcal{P}\mathrm{r}^{\gd}_{(2),0 \leq}}

\DeclareMathOperator{\Sp}{Sp}
\newcommand{\syntwo}{\mathrm{Syn}_{ \mathbb{F}_2}}

\newcommand{\xto}{\xrightarrow}
\newcommand{\tauinv}{\tau^{- 1}}

\author{Sophus Valentin Willumsgaard}
\title{Obstructions for Associativity in Stable Homotopy Theory}
\date{\today}

\begin{document}

\maketitle
\begin{abstract}
	We give a construction of the obstruction theory for \(\mathbb{A}_{n}\)-algebra structures
	in stable \(\infty\)-categories, and give some properties of it.
	We use this to show that the spectrum \(\mathbb{S} / 4\)
	admits an \(\mathbb{A}_5\)-multiplication using synthetic spectra.
\end{abstract}
\tableofcontents
\section{Introduction}
In higher algebra,
a classic problem is what multiplicative structures
there exist on quotients of the sphere spectrum.
In the discrete analogue,
all abelian groups
\( \Z / n \Z \)
admit unique commutative multiplications,
so one would expect the spectra \( \s / n \)
admit \( \ei \)-ring structures,
which is the higher analogue of commutative multiplications.
However what algebraic structures cofibers admit in higher algebra
are much more complicated,
depending highly on what number the cofiber is taken of.
The following results summarise the current knowledge of this
problem:%
\footnote{
	Proofs of these statements can be found in
	\cite{Angeltveit},\cite{Burklund} and \cite{Prasit} respectively.
}
\begin{enumerate}
	\item  For a prime \( p \), the Moore spectrum
	      \( \s / p \) admits an
	      \( \an[p-1] \)-algebra structure
	      but not an
	      \( \an[p] \)-algebra structure.
	\item For
	      \( q \geq \frac{3}{2} (n + 1 )  \),
	      the Moore spectrum
	      \( \s / 2^q\) admits an
	      \( \enn\)-algebra structure.
	      For \( p \) odd and \( q \geq n + 1 \),
	      the Moore spectrum \( \s / p ^ q \)
	      admits an \( \enn \)-algebra structure.
	\item  The Moore spectrum \( \s / 4 \) admits a
	      \( \an[4] \)-algebra structure,
	      but not an \( \E_2 \)-algebra structure.
\end{enumerate}

While the first two results are quite strong,
it is still an open question whether \( \s / 4 \)
admits an \( \E_1 \)-algebra structure.
The goal of the paper is to give an improvement of the current result:

\begin{thmx}[\Cref{main result}]
	The Moore spectrum \( \s / 4 \)
	admits an \( \mathbb{A}_5\)-algebra structure.
\end{thmx}
In \Cref{anoperads}
we recall the definition of \(\an\)-algebras.
In \Cref{secobst}
we construct an obstruction theory for \(\an\)-algebra structures on an
object in a monoidal stable \(\infty\)-category.
Such an obstruction theory is not new,
with the first exposition given by Alan Robinson in \cite{Robinson}
for \(\an\)-ring spectra,
however we extend this to any monoidal stable \(\infty\)-category.
To do so,
we first in \Cref{theclaim}
prove a general result about the relation between \(\an\)- and
\(\an[n-1]\)-algebras in a monoidal \(\infty\)-category,
which we then prove gives an obstruction theory,
in the stable setting:
\begin{thmx}[\Cref{obstruction}]
	Given a monoidal stable \( \infty \)-category \( \mathcal{C} \)
	and a map
	\( r: X \to 1_{ \mathcal{C} } \),
	there exists a sequence of inductively defined obstructions
	\begin{align*}
		\theta_k
		\in
		[\Sigma^{2k-3} X^{ \otimes k },
			1_{\mathcal{C} }/r],
	\end{align*}
	such that the vanishing of
	\( \theta_1, \dots, \theta_n \)
	induces a \( \an \)-algebra structure on
	\( 1_{ \mathcal{C} } /r \)
	with unit given by the cofiber map.
\end{thmx}
In the following two sections, we give conditions for when these
obstructions vanish.
In \Cref{freefun},
we apply the obstruction theory to universal objects, yielding the following
theorem:
\begin{thmx}[\Cref{coolthing}]
	Let \(\mathcal{C}\) be a 2-local presentable symmetric monoidal stable \( \infty \)-category,
	\( X \in \pic ( \mathcal{C} ) \)
	and \( v: X \to \textbf{1}_{ \mathcal{C} } \)
	be a map to the unit.
	If the map
	\( Q_1(v): \Sigma X^{\otimes 2} \to \textbf{1}_{ \mathcal{C} } \)
	vanishes,
	then \( 1/v \) admits a homotopy associative multiplication.
\end{thmx}
In \Cref{obstructionmapp},
we establish a naturality of our obstruction theory,
so that a monoidal functor \(F\),
will send the obstruction \(\theta_k\)
on an \(\mathbb{A}_{k-1}\)-algebra \(A\)
to the obstruction for the induced
\(\mathbb{A}_{k-1}\)-algebra structure on \(F(A)\).
Using these techniques,
we are able to show in \Cref{synthetic}
that \(\mathbb{S}/4\) admits an \(\an[5]\)-algebra structure.

\subsection*{Notations and Conventions}
We will use the setting and language of \(\infty\)-categories developed
by Jacob Lurie in \cite{HTT} and \cite{HA}. We use the grading convention
for synthetic spectra used in \cite{Burklund}.

%
\subsection*{Acknowledgements}
I would like to thank my advisor Robert Burklund for his guidance in writing
this paper.

I would like to thank Jan Steinebrunner, Qingyuan Bai, Vignesh Subramanian, Maxime Ramzi \& Yu Leon Liu
for helpful discussions about different topics related to the thesis,
and to
Marius Verner Bach Nielsen and Andreas Momme Studsgaard
for giving helpful feedback for drafts of my thesis.

\section{\texorpdfstring{\(\mathbb{A}_n\)}{An}-Operads}\label{anoperads}
There are different notions of an
associative algebra structure on a topological space with a unit map
\( e: * \to  X \).
A topological monoid consists of a map
\( \mu : X \times X \to X \),
that is strictly associative and unital by satisfying the following identities
\begin{gather*}
	\mu ( \mu \times 1 )
	=
	\mu ( 1 \times \mu)
	\\
	\mu(e \times \id_X)
	=
	\mu(\id_X \times e)=\id_X.
\end{gather*}
Another choice is an associative H-space,
which instead of requiring the maps to strictly agree,
only requires them to be homotopic
\begin{gather*}
	[\mu ( \mu \times 1 )]
	=
	[\mu ( 1 \times \mu)]
	\\
	[\mu(e \times \id_X)]
	=
	[\mu(\id_X \times e)]
	=
	[\id_X].
\end{gather*}
Further,
an \( \an[3] \)-algebra is informally an associative H-space
together with a choice of such homotopies.
In homotopy theory,
associative H-space structures are more natural,
when spaces are only defined up to homotopy.
Associative H-spaces however,
do not have the same useful properties as topological monoids have.
\begin{remark}
	The \(\an[3]\)-operad is not coherent in the sense of \cite{HA},
	which is used for giving a
	well-defined tensor product on the module category
	of an algebra.
\end{remark}

The problems with an associative H-space,
can be seen with multiplication of four elements.
Each way of ordering the multiplication,
gives a point in the space of maps
\( \map_{\textbf{Top}}(X^{ \times 4}, X) \).
Since the different ordering of multiplications are homotopic,
we can choose paths in the mapping space connecting the different points,
as can be seen in the following diagram.

\begin{figure}[H]
	\centering
	\begin{tikzpicture}[mystyle/.style={draw,shape=circle,fill=white,
					inner sep=0pt, minimum size=4pt,
					label={[anchor=center, label distance=2mm](90+360/\ngon*(#1-1)):#1}}]
		\def\ngon{5}
		\node[draw, regular polygon,regular polygon sides=\ngon,minimum size=5cm] (p) {};
		\node[circle,radius=.01cm,draw,
		label=above:{$(ab)(cd)$},
		fill=white] at (p.corner 1) {};
		\node[circle,radius=.01cm,draw,
			label={[xshift=-0.3cm, yshift=0.1cm]$((ab)c)d$},
			fill=white] at (p.corner 2) {};
		\node[circle,radius=.01cm,draw,
		label=below:{$(a(bc))d$},
		fill=white] at (p.corner 3) {};
		\node[circle,radius=.01cm,draw,
		label=below:{$a((bc)d)$},
		fill=white] at (p.corner 4) {};
		\node[circle,radius=.01cm,draw,
			label={[xshift=0.3cm, yshift=0.1cm]$a(b(cd))$},
			fill=white] at (p.corner 5) {};
	\end{tikzpicture}
	\caption{Space of multiplications of four elements.}
\end{figure}
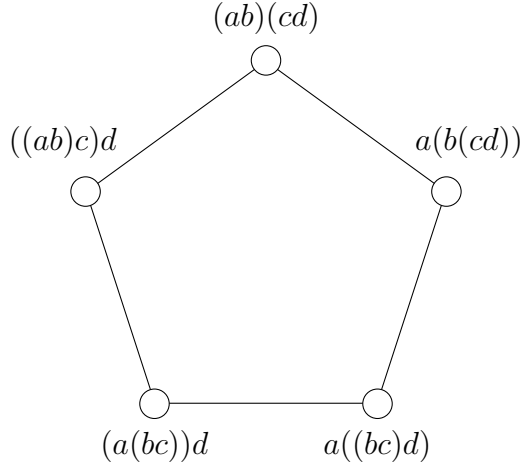

From this picture we see that the space of multiplications of four elements,
might not be contractible.
If we want to have a unique multiplication of four elements up to contractible choice,
we need a nulhomotopy of this loop.
We can continue this inductively filling out homotopies in
\( \map_{\textbf{Top}}(X^{\times n }, X) \).
We then get different
associative multiplicative structures on \( X \),
depending on how many higher homotopies we require.

This leads to the definition of the \(\an\)-operad,
which can be defined from its universal property following
\cite{restrict}.
\begin{theorem}\label{ffadjoints}
	\cite[Theorem 1.2]{restrict}
	Let \(\unOp\) be the \(\infty\)-category of unital \(\infty\)-operads
	given as complete Segal presheaves on closed trees \(\Segc\),
	and \(\unOp_{\leq k}\) denote the \(\infty\)-category of
	\(k\)-restricted unital \(\infty\)-operads
	given by complete Segal presheaves on closed \(k\)-dendroidal trees \(\Segck\).
	Then the restriction functor
	\((-)^{k}: \unOp \to \unOp_{\leq k}\)
	has both a fully faithful left adjoint \(L_{k}\)
	and a fully faithful right adjoint \(R_{k}\),
	both given by Kan extension.
\end{theorem}
\begin{definition}\label{bounddef}
	Given a unital \(\infty\)-operad \(\mathcal{O}\) and a color \(X\), we
	define
	\begin{align*}
		\bound_{n}^{X}(\mathcal{O})   & := \mul_{L_{n}(\mathcal{O}^{n})}(X^{(n)},X)   \\
		\cobound_{n}^{X}(\mathcal{O}) & := \mul_{R_{n}(\mathcal{O}^{n})}(X^{(n)}, X).
	\end{align*}
	If \(\mathcal{O}\) has a contractible space of colors,
	we suppress the color from the notation.
\end{definition}
\begin{example}[\cite{restrict}]
	We define \(\an := L_{n}((\ai)^{n})\).
	Let \(K_{n}\) denote the \(n\)-th Stasheff associahedron,
	a contractible \((n-2)\)-dimensional polytope with boundary \(\partial K_{n}\).
	We have \(\an(n) \simeq \ai(n) \simeq K_{n} \times \Sigma_{n}\),
	and from \cite[Example 3.1.13]{Florian}
	\(\an[n-1](n) = \bound_{n-1}\ai = \partial K_{n}\),
	matching our expectations from above.
\end{example}

\begin{remark}
	From this we see that \( \an \)-algebras are like
	\( \ai\)-algebras, where only multiplication of up to \(
	n \) elements are fully defined.
	An alternative definition based on Lurie's definition of
	\(\infty\)-operads is given in \cite[Remark 4.1.4.8]{HA}.
\end{remark}

\section{Extending \texorpdfstring{\(\an\)}{an}-structures}\label{secobst}
We are interested in what structure is needed to extend an \(\an[n-1]\)-algebra
to an \(\an\)-algebra.
For the non-unital case,
Lurie gives the following theorem:

\begin{theorem}[\cite{HA}, Theorem 4.1.6.8]\label{lurietheorem}
	Let \(\mathcal{C}\) be a monoidal \(\infty\)-category,
	let \(A\) be an object of \(\mathcal{C}\),
	and let \(n \geq 2\).
	Then there is a pullback diagram of \(\infty\)-categories
	\[\begin{tikzcd}
			{\alg^{\Nu}_{\mathbb{A}_n}( \mathcal{C})\times_{\mathcal{C}}\{A\}}
			& {\map_{\mathcal{C}}(A^{\otimes n},A)^{K_n}} \\
			{\alg^{\Nu}_{\mathbb{A}_{n-1}}( \mathcal{C})\times_{\mathcal{C}}\{A\}}
			& {\map_{\mathcal{C}}(A^{\otimes n},A)^{\partial K_n},}
			\arrow[from=1-1, to=2-1]
			\arrow[from=1-1, to=1-2]
			\arrow["\beta",from=2-1, to=2-2]
			\arrow[from=1-2, to=2-2]
			\arrow["\ulcorner"{anchor=center, pos=0.125}, draw=none, from=1-1, to=2-2]
		\end{tikzcd}\]
	where
	\( \alg^{\Nu}_{\mathbb{A}_n}(\mathcal{C}) \)
	is the \(\infty\)-category of
	non-unital \(\mathbb{A}_n\)-algebras in \(\mathcal{C}\).
\end{theorem}

The data to extend
a nonunital \(\mathbb{A}_{n-1}\)-algebra to
a nonunital \(\an\)-algebra,
is then a multiplication map
\(\mu_n : A^{\otimes n} \to A\),
which is compatible with the \(\mathbb{A}_{n-1}\)-structure.
A similar result should hold for unital algebras,
but the \(n\)-fold multiplication
should also be compatible with the unital structure.
The coboundary object from \Cref{bounddef} lets us keep
track of this data. We then extend this construction to non-unital
\(\infty\)-operads.
\begin{definition}
	Fix \(n \geq 1\).
	Let \(\mathcal{P}(n)\) be the category of subsets of \(\{1,..., n\}\) with
	inclusions, and let \(\mathcal{P}(n)_{\leq k}\) be the subcategory
	consisting of subsets \(I\) with \(|I| \leq k \).

	Given an object in the underlying \(\infty\)-category of an operad \(X \in
	\mathcal{O}\)	and a unit
	\(i \in \mul_{\mathcal{O}}(\varnothing,X)\),
	consider the contravariant functor \(F\)
	that sends every inclusion of subset
	\(J \subseteq I \subseteq \{1,..., n\}\) to the map
	\(\mul_{\mathcal{O}}(X^{(|I|)}, X) \to \mul_{\mathcal{O}}(X^{(|J|)}, X)\)
	by composing with the unit map \(i\) in every index not in \(J\).
	The coboundary is then defined as
	\begin{align*}
		\cobound^{(X,i)}_{n} (\mathcal{O}) :=
		\lim_{\mathcal{P}(n)^{\op}_{\leq n-1}} F.
	\end{align*}
	When the choice of unit is clear,
	we will obscure it from the notation.
	By \cite[Lemma 5.4]{restrict},
	this agrees with the definition for unital \(\infty\)-operads.
\end{definition}
We will then prove the following unital version of \Cref{lurietheorem}.

\begin{theorem}\label{theclaim}
	Let \(\mathcal{C}\) be a monoidal \(\infty\)-category,
	\(A\) an object of \(\mathcal{C}\),
	\(\varphi: \mathbf{1}_{\mathcal{C}} \to A\)
	a map from the unit
	and \(n \geq 2\).
	Then there is a natural pullback diagram
	\[\begin{tikzcd}
			{\alg_{\mathbb{A}_n}(\mathcal{C})\times_{\mathcal{C}_{1/}}
				\{1 \xrightarrow{u} A\}}
			&
			{\map_{\mathcal{C}}(A^{\otimes n},A)^{K_n}}
			\\
			{\alg_{\mathbb{A}_{n-1}}(\mathcal{C})\times_{\mathcal{C}_{1/}}
			\{1 \xrightarrow{u} A\}}
			&
			{\map_{\mathcal{C}}(A^{\otimes n},A)^{\partial K_n}
					\times_{(\cobound_{n}^{A} (\mathcal{C}))^{\partial K_n}}
					(\cobound_{n}^{A}(\mathcal{C}))^{K_n}.}
			\arrow[from=1-1, to=2-1]
			\arrow[from=1-1, to=1-2]
			\arrow[from=2-1, to=2-2]
			\arrow[from=1-2, to=2-2]
		\end{tikzcd}\]
\end{theorem}
To prove this,
we will use some results of reduced \(\infty\)-operads.
\begin{definition}
	An \(\infty\)-operad is reduced if it is unital and its underlying
	\(\infty\)-category is a contractible space. We denote the full
	subcategory of \(\Op\) spanned by the reduced \(\infty\)-operads by
	\(\redOp\).\footnote{This is not the same definition as the one given in \cite{Ieke}.}
\end{definition}
\begin{theorem}\label{Floriansquare}(\cite[Remark 3.1.22]{Florian})
	Let \(\mathcal{O},\mathcal{W}\) be reduced \(\infty\)-operads.
	We then have a pullback diagram
	\[\begin{tikzcd}
			{\map_{\unOp_{\leq n}}(\mathcal{O}^{n},\mathcal{W}^{n})}
			&
			{\map^{\Sigma_{k}}_{\mathcal{S}}(\mathcal{O}(n),\mathcal{W}(n))}
			\\
			{\map_{\unOp_{\leq n-1}}(\mathcal{O}^{n-1},\mathcal{W}^{n-1})}
			&
			{T,}
			\arrow[from=1-1, to=1-2]
			\arrow[from=1-1, to=2-1]
			\arrow[from=1-2, to=2-2]
			\arrow[""{name=0, anchor=center, inner sep=0}, from=2-1, to=2-2]
			\arrow["\lrcorner"{anchor=center, pos=0.125}, draw=none, from=1-1, to=0]
		\end{tikzcd}\]
	where
	\begin{align*}
		T:= \map^{\Sigma_{k}}_{\mathcal{S}}(\bound_{n} \mathcal{O},\mathcal{W}(n))
		\times_{\map^{\Sigma_{k}}_{\mathcal{S}}(\bound_{n} \mathcal{O},\cobound_n \mathcal{W})}
		\map^{\Sigma_{k}}_{\mathcal{S}}(\mathcal{O}(n),\cobound_n \mathcal{W}).
	\end{align*}
\end{theorem}
\begin{remark}
	This theorem is originally stated in terms of the model category of
	complete Segal closed dendroidal spaces \(\mathrm{Seg}^{\mathrm{cpl}}(\Omega^{c})\),
	but the underlying \(\infty\)-category
	is equivalent to the \(\infty\)-category of unital \(\infty\)-operads
	by \cite[Corollary 3.20]{restrict}.
\end{remark}
We want to apply the theorem, to the endomorphism
operad, however the endomorphism operad is not reduced.
To remedy this,
we construct a functor sending an \(\infty\)-operad with a map from the unit
to a reduced \(\infty\)-operad. This is same construction as in \cite{EH},
but extended to non-unital \(\infty\)-operads.

\subsection{Reduced Endomorphism Operad}
By \cite[Lemma 2.2.4]{EH}, \(\enul\) is initial in \(\redOp\),
so we get a functor
\begin{align*}
	\redOp \simeq (\redOp)_{\enul /} \to (\unOp)_{\enul/} \to (\Op)_{\enul/}
\end{align*}

\begin{lemma}
	The inclusion \((\unOp)_{\enul/} \to (\Op)_{\enul/}\) has a right
	adjoint.
\end{lemma}
\begin{proof}

	By \cite[Proposition 2.3.1.11]{HA}, the inclusion \(\unOp \to \Op\)
	has a right adjoint, which sends an operad \(\mathcal{O}^{\otimes }\) to the
	\(\infty\)-operad \(\mathcal{O}_{*}^{\otimes }\) of
	pointed objects of \(\mathcal{O}^{\otimes }\) (note that by \cite[Lemma 2.3.1.5]{HA} the final object of
	an \(\infty\)-operad is always an object lying over \(\langle 0\rangle \in
	N(\mathrm{Fin}_{*})\), so a pointed object is a map from a unit).
	Then by \cite[Lemma 2.1.2]{EH} we get that the slice functor
	\begin{align*}
		(\unOp)_{\enul/} \to (\Op)_{\enul/}
	\end{align*}
	also has a right adjoint.
\end{proof}
In \cite[Proposition 2.2.9]{EH}
they construct a right adjoint to the functor
\begin{align*}
	\redOp \to (\unOp)_{\enul/},
\end{align*}
leading to the following proposition.
\begin{proposition}
	The functor \(\redOp \to (\Op)_{\enul/}\) has a right adjoint
	\((-)^{\red}\).\end{proposition}
\begin{notation}
	By \cite[Proposition 2.1.3.9]{HA} a map \(\enul \to \mathcal{O}\) is
	equivalent to a color \(X\)
	along with a unit map in \(\mul(\varnothing,X)\).
	For such a map we denote \(\mathcal{O}^{\red}\)
	by \(\End^{\red}_{\mathcal{O}}(X)\).
\end{notation}

We will now describe the multi mapping spaces of
\(\End_{\mathcal{O}}^{\red}(X)\).

\begin{proposition}\label{fiberseq}
	Let \(\mathbb{E}_{0} \to \mathcal{O}\) be a map of \(\infty\)-operads determined by
	\(X \in {\mathcal{O}}\)
	and a unit map
	\(\mul_{\mathcal{O}}(\varnothing,X)\).
	Then for every \(n \in \mathbb{N}\),
	we have fiber sequences
	\[\begin{tikzcd}
			{\End_{\mathcal{O}}^{\red}(X)(n)} & {\mul(X^{(n)},X)} & P_{n}
			\arrow[from=1-1, to=1-2]
			\arrow[from=1-2, to=1-3]
		\end{tikzcd}\]
	\[\begin{tikzcd}
			{\cobound_{n}\End_{\mathcal{O}}^{\red}(X)} & {\cobound^{X}_{n}\mathcal{O}} & P_{n},
			\arrow[from=1-1, to=1-2]
			\arrow[from=1-2, to=1-3]
		\end{tikzcd}\]
	where \(P_{n}\) is limit of
	\[\begin{tikzcd}
			{\map_{\mathcal{O}}(X,X)} & {\cdots\; n \text{ copies }\cdots} & {\map_{\mathcal{O}}(X,X)} \\
			& {\mul_{\mathcal{O}}(\varnothing,X).}
			\arrow[from=1-1, to=2-2]
			\arrow[from=1-2, to=2-2]
			\arrow[from=1-3, to=2-2]
		\end{tikzcd}\]
\end{proposition}

\begin{proof}
	From Proposition 2.2.9 and Lemma 2.2.11 in \cite{EH}, the \(n\)'th mapping
	space is given by the pullback
	\[\begin{tikzcd}
			{\End_{\mathcal{O}}^{\red}(X)(n)} & {\mul_{\mathcal{O_{*}}}(X^{(n)},X)} \\
			{\ei(n)} & {\map_{\mathcal{O}_{*}}(X,X)^{\times n}.}
			\arrow[from=1-1, to=1-2]
			\arrow[from=1-1, to=2-1]
			\arrow["\lrcorner"{anchor=center, pos=0.125}, draw=none, from=1-1, to=2-2]
			\arrow[from=1-2, to=2-2]
			\arrow[from=2-1, to=2-2]
		\end{tikzcd}\]
	The terms in the diagram are given by
	\begin{align*}
		\ei(n)                                  & \simeq *                           \\
		\mul_{\mathcal{O}_{*}}(X^{(n)}, X)      &
		\simeq
		\mul_{\mathcal{O}}(X^{(n)}, X) \times_{\mul_{\mathcal{O}}(\varnothing, X)} * \\
		\map_{\mathcal{O}_{*}}(X, X)^{\times n} & \simeq
		\mul_{\mathcal{O}}(X, X)^{\times n} \times_{\mul_{\mathcal{O}}(\varnothing, X)^{\times n}} *,
	\end{align*}
	so we can rewrite the diagram as
	\[\begin{tikzcd}
			{\End_{\mathcal{O}}^{\red}(X)(n)}	&{\mul_{\mathcal{O}}(X^{(n)}, X) \times_{\mul_{\mathcal{O}}(\varnothing, X)} *}
			\\
			{*} & {\mul_{\mathcal{O}}(X, X)^{\times n} \times_{\mul_{\mathcal{O}}(\varnothing, X)^{\times n}} *.}
			\arrow[from=1-1, to=1-2]
			\arrow[from=1-1, to=2-1]
			\arrow["\lrcorner"{anchor=center, pos=0.125}, draw=none, from=1-1, to=2-2]
			\arrow[from=1-2, to=2-2]
			\arrow[from=2-1, to=2-2]
		\end{tikzcd}\]
	We have another pullback diagram with the two most right terms
	\[\begin{tikzcd}
			{\mul_{\mathcal{O}}(X^{(n)}, X) \times_{\mul_{\mathcal{O}}(\varnothing, X)} *} &
			{\mul_{\mathcal{O}}(X^{(n)}, X)}
			\\
			{\mul_{\mathcal{O}}(X, X)^{\times n} \times_{\mul_{\mathcal{O}}(\varnothing, X)^{\times n}} *}&
			{\mul_{\mathcal{O}}(X, X)^{\times n} \times_{\mul_{\mathcal{O}}(\varnothing, X)^{\times n}} \mul_{\mathcal{O}}(\varnothing, X),}
			\arrow[from=1-1, to=1-2]
			\arrow[from=1-1, to=2-1]
			\arrow["\lrcorner"{anchor=center, pos=0.125}, draw=none, from=1-1, to=2-2]
			\arrow[from=1-2, to=2-2]
			\arrow[from=2-1, to=2-2]
		\end{tikzcd}\]
	where the lower right corner is equivalent to \(P_{n}\).
	Pasting these squares together we get
	\[\begin{tikzcd}
			{\End_{\mathcal{O}}^{\red}(X)(n)} & {\mul_{\mathcal{O}}(X^{(n)}, X)}
			\\
			{*} & {P_{n}.}
			\arrow[from=1-1, to=1-2]
			\arrow[from=1-1, to=2-1]
			\arrow["\lrcorner"{anchor=center, pos=0.125}, draw=none, from=1-1, to=2-2]
			\arrow[from=1-2, to=2-2]
			\arrow[from=2-1, to=2-2]
		\end{tikzcd}\]
	For the second fiber sequence,
	note that the coboundary preserves limits,
	so we also get a limit diagram
	\[\begin{tikzcd}
			{\cobound_{n}\End_{\mathcal{O}}^{\red}(X)} &
			{\cobound^{X}_{n} \mathcal{O}_{*}} \\
			{\cobound_{n}\ei} & {\map_{\mathcal{O}_{*}}(X,X)^{\times n}}
			\arrow[from=1-1, to=1-2]
			\arrow[from=1-1, to=2-1]
			\arrow["\lrcorner"{anchor=center, pos=0.125}, draw=none, from=1-1, to=2-2]
			\arrow[from=1-2, to=2-2]
			\arrow[from=2-1, to=2-2]
		\end{tikzcd}\]
	Further inserting
	the limit formula in the upper right corner we get the diagram
	\[\begin{tikzcd}
			{\cobound_{n}\End_{\mathcal{O}}^{\red}(X)} &
			{\cobound^{X}_{n} \mathcal{O} \times_{\mul_{\mathcal{O}}(\varnothing,X)} *} \\
			{*} & {\map_{\mathcal{O}_{*}}(X,X)^{\times n}.}
			\arrow[from=1-1, to=1-2]
			\arrow[from=1-1, to=2-1]
			\arrow["\lrcorner"{anchor=center, pos=0.125}, draw=none, from=1-1, to=2-2]
			\arrow[from=1-2, to=2-2]
			\arrow[from=2-1, to=2-2]
		\end{tikzcd}\]
	Then by the same argument as in the first fiber sequence we get the
	pullback square
	\[\begin{tikzcd}
			{\cobound_{n}\End_{\mathcal{O}}^{\red}(X)} &
			{\cobound^{X}_{n} \mathcal{O}} \\
			{*} & {P_{n}.}
			\arrow[from=1-1, to=1-2]
			\arrow[from=1-1, to=2-1]
			\arrow["\lrcorner"{anchor=center, pos=0.125}, draw=none, from=1-1, to=2-2]
			\arrow[from=1-2, to=2-2]
			\arrow[from=2-1, to=2-2]
		\end{tikzcd}\]
\end{proof}
\subsection{Proof of \texorpdfstring{\Cref{theclaim}}{the Claim}}
We are now ready to prove the main theorem of the section:
\begin{proof}[Proof of \Cref{theclaim}]
	We consider the pullback square in \Cref{Floriansquare}
	with the reduced \(\infty\)-operads
	\(\mathcal{O} = \mathbb{A}_{\infty}\) and
	\(\mathcal{W} = \End_{\mathcal{O}}^{\red}(A)\).
	We then have
	\begin{gather*}
		\map_{\redOp_{\leq n}}((\ai)^{n},\End_{\mathcal{C}}^{\red}(A)^{n})
		\simeq
		\map_{\redOp}(\an,\End_{\mathcal{C}}^{\red}(A))
		\simeq
		\alg_{\an}(\mathcal{C}) \times_{\mathcal{C}_{1/}} \{ 1 \xrightarrow{u} A\} \\
		\mathbb{A}_{\infty}(n)\simeq K_{n} \times \Sigma_{n} \\
		\bound_{n}\mathbb{A}_{\infty}\simeq \partial K_{n} \times \Sigma_{n}.
	\end{gather*}
	Since the action of the symmetric group on \(\ai\) is free,
	we can consider the mapping spaces without the symmetric groups
	\begin{align*}
		\map^{\Sigma_{k}}(\mathbb{A}_{\infty} \times K_{n},
		\End_{\mathcal{C}}^{\red}(A))
		\simeq
		\map(\mathbb{A}_{\infty},
		\End_{\mathcal{C}}^{\red}(A)).
	\end{align*}
	This then gives us the diagram
	\[\begin{tikzcd}
			{\alg_{\mathbb{A}_n}(\mathcal{C})\times_{\mathcal{C}_{1/}}
				\{1 \xrightarrow{u} A\}} & {\End_{\mathcal{C}}^{\red}(A)(n)^{K_n}} \\
			{\alg_{\mathbb{A}_{n-1}}(\mathcal{C})\times_{\mathcal{C}_{1/}}
			\{1 \xrightarrow{u} A\}}
			&
			{T(\End_{\mathcal{C}}^{\red}(A))}
			\arrow[from=1-1, to=1-2]
			\arrow[from=1-1, to=2-1]
			\arrow[from=1-2, to=2-2]
			\arrow["\beta", from=2-1, to=2-2]
		\end{tikzcd}\]
	with
	\begin{align*}
		T(\mathcal{O}):=\mathcal{O}(n)^{\partial K_n}
		\times_{\cobound_{n} \mathcal{O}^{\partial K_n}}
		\cobound_n \mathcal{O}^{K_n}.
	\end{align*}
	Using the fiber sequences from \cref{fiberseq},
	we get that
	\(T(\End_{\mathcal{C}}^{\red}(A))\) is equivalent to

	\begin{align*}
		\fib\left(\map_{\mathcal{C}}(A^{\otimes  n}, A)^{\partial K_{n}} \to
		P_{n}^{\partial K_{n}}\right)
		\times_{\fib\left(\cobound^{A}_{n}(\mathcal{C})
			\to P_{n}^{\partial K_{n}} \right)}
		\fib\left(\cobound^{A}_{n}(\mathcal{C})^{K_{n}} \to
		P_{n}^{K_{n}}\right).
	\end{align*}
	By commuting these limits this reduces to
	\begin{align*}
		\fib\left(T(\End_{\mathcal{C}}(A))\to 	P_{n}^{K_{n}}\right).
	\end{align*}
	From this we get another pullback square
	\[\begin{tikzcd}
			{\End_{\mathcal{C}}^{\red}(A)(n)^{K_{n}}}
			&
			{\map_{\mathcal{C}}(A^{\otimes n}, A)^{K_{n}}}
			\\
			{T(\End_{\mathcal{C}}^{\red}(A))}
			&
			{T(\End_{\mathcal{C}}(A)),}
			\arrow[from=1-1, to=1-2]
			\arrow[from=1-1, to=2-1]
			\arrow[from=1-2, to=2-2]
			\arrow[""{name=0, anchor=center, inner sep=0}, from=2-1, to=2-2]
			\arrow["\lrcorner"{anchor=center, pos=0.125}, draw=none, from=1-1, to=0]
		\end{tikzcd}\]
	which when pasted with the previous pullback square give the wanted square
	\[\begin{tikzcd}
			{\alg_{\mathbb{A}_n}(\mathcal{C})\times_{\mathcal{C}_{1/}}
				\{1 \xrightarrow{u} A\}}
			&
			{\map_{\mathcal{C}}(A^{\otimes n},A)^{K_n}}
			\\
			{\alg_{\mathbb{A}_{n-1}}(\mathcal{C})\times_{\mathcal{C}_{1/}}
			\{1 \xrightarrow{u} A\}}
			&
			{\map_{\mathcal{C}}(A^{\otimes n},A)^{\partial K_n}
					\times_{\cobound^{A}_{n}(\mathcal{C})^{\partial K_n}}
					\cobound^{A}_{n}(\mathcal{C})^{K_n}.}
			\arrow[from=1-1, to=2-1]
			\arrow[from=1-1, to=1-2]
			\arrow["\beta", from=2-1, to=2-2]
			\arrow[from=1-2, to=2-2]
		\end{tikzcd}\]
\end{proof}

\subsection{Obstruction Theory}

We will show how \Cref{theclaim}
leads to an obstruction theory when \( \mathcal{C}
\) is a presentably monoidal stable \(\infty\)-category.
\begin{lemma}
	Let
	\(\mathcal{C} \in \alg(\prst)\)
	be a
	presentably monoidal stable \(\infty\)-category,
	so that the tensor product commutes with small colimits,
	and let \(u: 1_{\mathcal{C}} \to A\) be a map from the unit to an object.
	We then have a fiber sequence of spaces
	\[\begin{tikzcd}
			\map_{\mathcal{C}}(A^{\otimes n}, A)^{K_{n}} &
			{\map_{\mathcal{C}}(A^{\otimes n},A)^{\partial K_n}
					\times_{\cobound^{A}_{n}(\mathcal{C})^{\partial K_n}}
					\cobound^{A}_{n}(\mathcal{C})^{K_n}} \\
			*&\map_{\mathcal{C}}(\Sigma^{n-3}(A/u)^{\otimes n}, A).
			\arrow[from=1-1, to=1-2]
			\arrow[from=1-1, to=2-1]
			\arrow[from=2-1, to=2-2]
			\arrow["\lrcorner"{anchor=center, pos=0.125}, draw=none, from=1-1, to=2-2]
			\arrow[from=1-2, to=2-2]
		\end{tikzcd}\]
\end{lemma}
\begin{proof}
	Since
	\(\mathcal{C}\)
	is a presentably stable \(\infty\)-category,
	it is left-tensored over \(\Sp\)
	and further enriched over \(\Sp\) by \cite[Proposition 4.2.1.33]{HA}.
	In this case we have equivalences of spectra
	\begin{align*}
		\map_{\mathcal{C}}( A^{\otimes n},A )^{K_n}
		\simeq
		\map_{\mathcal{C}}( \Sigma^{ \infty } K_n \otimes  A^{\otimes n},A ) \\
		\cobound_{n}^{A}(\mathcal{C}) \simeq
		\map_{\mathcal{C}}\left( \colim_{I \in \mathcal{P}(n)^{\op}_{\leq n-1}}A^{\otimes |I|}, A\right).
	\end{align*}
	We denote the colimit inside by \(A^{\cob}\).
	The upper horizontal map in the sequence is then pre-composition with the map
	\begin{align*}
		\Sigma^{\infty} \partial K_n \otimes  A^{\otimes n}
		\amalg_{\Sigma^{\infty} \partial K_n \otimes A^{\cob}}
		\Sigma^{\infty} K_n \otimes A^{\cob}
		\to
		\Sigma^{\infty} K_n \otimes  A^{\otimes n}.
	\end{align*}
	We can then calculate the fiber of this map
	\begin{align*}
		\Sigma^{-1}\left(\Sigma^{\infty} (K_n / \partial K_n) \otimes
			(A^{\otimes n}/A^{\cob})\right)
		\cong
		\mathbb{S}^{n-3} \otimes (A/u)^{\otimes n}
		\simeq \Sigma^{n-3} (A/u)^{\otimes n}.
	\end{align*}
	From this we get the cofiber of the horizontal map in the sequence
	is given by
	\begin{align*}
		\map_{\mathcal{C} }( \Sigma^{n-3} (A/u)^{\otimes n},A).
	\end{align*}
	It follows that the above sequence is a fiber/cofiber sequence in spectra,
	and so is a fiber sequence in spaces, as the functor \(\Omega^{\infty}\)
	is limit-preserving.
\end{proof}
We can now construct the obstruction theory.
\begin{proposition}\label{obstruction}
	Given \(\mathcal{C} \in \alg(\prst)\)
	and an \(\an[n-1]\)-algebra \(A\) for \(n \geq 2\)
	there exists an obstruction
	\begin{align*}
		\theta_n \in [\Sigma^{n-3} (A/u)^{ \otimes n } , A],
	\end{align*}
	such that \(\theta_{n}\) is nulhomotopic if and only if
	\(A\) can be extended to an \(\mathbb{A}_n\)-algebra.
	If \(A\) is the cofiber of a map \(X \to 1_{\mathcal{C}}\)
	then the obstruction lies in
	\begin{align*}
		\theta_n \in [\Sigma^{2n-3} (X)^{ \otimes n } , A].
	\end{align*}
\end{proposition}
\begin{proof}
	Given an \( \mathbb{A}_{n-1} \)-algebra \( A \) in \(\mathcal{C}\)
	with unit \( u: 1_{ \mathcal{C} } \to A \),
	we have a diagram
	\[\begin{tikzcd}
			{\alg_{\mathbb{A}_n}( \mathcal{C})\times_{\mathcal{C}_{1/}}\{1
				\xrightarrow{u} A\}} & {\map(A^{\otimes n},A)^{K_n}} \\
			{\alg_{\mathbb{A}_{n-1}}( \mathcal{C})
			\times_{\mathcal{C}_{1/}}
			\{1
			\xrightarrow{u} A\}} & {\map(A^{\otimes n},A)^{\partial K_n}
					\times_{\cobound_n^{A}(\mathcal{C})^{\partial K_n}}
					\cobound^{A}_n(\mathcal{C})^{K_n}.} \\
			{*}
			\arrow[from=1-1, to=1-2]
			\arrow[from=1-1, to=2-1]
			\arrow[from=1-2, to=2-2]
			\arrow["\beta", from=2-1, to=2-2]
			\arrow["A", from=3-1, to=2-1]
		\end{tikzcd}\]
	From \Cref{theclaim},
	giving an extension of \(A\) to an \(\an\)-algebra,
	is equivalent to a lift of \( \beta ( A ) \) to
	\( \map_ {\mathcal{C}} ( A^{\otimes n } , A ) ^ {K_n} \).
	Since
	\( \map_ {\mathcal{C}} ( A^{\otimes n } , A ) ^ {K_n} \)
	is the fiber of the map
	\begin{align*}
		\map_{\mathcal{C}}(A^{\otimes n},A)^{\partial K_n}
		\times_{\cobound_{n}^{A}(\mathcal{C})^{\partial K_n}}
		\cobound_{n}^{A}(\mathcal{C})^{ K_n}
		\xrightarrow{q}
		\map_{\mathcal{C} }( \Sigma^{n-3} (A/u)^{\otimes n},A),
	\end{align*}
	the space of extensions of \( A \) to an \( \an \)-algebra structure
	is equivalent to space of nulhomotopies of \(q \beta (A)\),
	so \( q \beta (A) \) is an obstruction for extending \(A\)
	to an \(\mathbb{A}_{n}\)-algebra.

\end{proof}

\begin{remark}
	Given an object \(A \in \mathcal{C}^{\otimes}\) with a map
	\(1_{\mathcal{C}} \to A\),
	it automatically has the structure of an \(\mathbb{A}_{1}\)-algebra,
	so the obstruction \(\theta_{2}\) is defined.
	If it vanishes,
	A can be given a structure of an \(\an[2]\)-algebra,
	and such a choice gives the obstruction \(\theta_{3}\).
	In this way, one can inductively construct obstructions \(\theta_{2},
	\dots, \theta_{n}\) for the existence of an \(\an\)-structure on an
	object.
\end{remark}

\section{Applying Obstruction Theory to the Universal Case}\label{freefun}
In the previous section we defined an obstruction theory,
for the existence of \(\mathbb{A}_{n}\)-structures.
In this section,
we construct a stable symmetric monoidal \(\infty\)-category,
classifying maps from strict elements to the unit.
We then apply the obstruction theory in this \(\infty\)-category,
to give general results on when the obstructions vanish.
\subsection{Locally Graded Stable	\texorpdfstring{\(\infty\)}{inf}-Categories}\label{localgrade}

We recall the notion of a locally graded stable \(\infty\)-category
from \cite{Rotation}, and adapt to the setting of \(\prl\).

\begin{proposition}\label{spnmap}
	Consider the \(\infty\)-category \(\Sp^{\mathbb{N}} \in \calg(\prst)\),
	with symmetric monoidal structure induced by \(\mathbb{N}\) from the
	stable Yoneda embedding, and let \(\mathcal{C}\) be a symmetric monoidal \(\infty\)-category.
	We then have an equivalence
	\begin{align*}
		\map_{\calg(\prst)}(\Sp^{\mathbb{N}}, \mathcal{C})\simeq
		\map_{\ei}(\mathbb{N}, \mathcal{C}).
	\end{align*}
\end{proposition}
\begin{proof}
	This follows from \cite[Proposition 4.8.1.10 \& Corollary 4.8.1.14]{HA}
	along with \cite[Proposition 2.3.7]{Rotation}.
\end{proof}
\begin{definition}
	The \(\infty\)-category \(\prgd\)
	of locally graded stable presentable \(\infty\)-categories
	is defined as
	\(\rmod_{\Sp^{\mathbb{N}}}\left(\prst\right)\).
\end{definition}
\begin{remark}\label{strictelement}
	Given a strict element
	\(C \in \map_{\ei}(\mathbb{N}, \mathcal{C})\),
	from \cref{spnmap} we get a symmetric monoidal functor
	\(\Sp^{\mathbb{N}} \to \mathcal{C}\),
	which then gives \(\mathcal{C}\) the structure of a commutative algebra
	in \(\rmod_{\Sp^{\mathbb{N}}}(\prst)\).
\end{remark}
\begin{notation}
	We let \(\mathbb{S}(n)\) denote the graded spectrum,
	with the \(n\)'th index given by \(\mathbb{S}\),
	and the rest being \(0\).
	In general for any object \(C\) in a
	\(\infty\)-category \(\mathcal{C} \in \prgd\),
	we define \(C(n):=\mathbb{S}(n) \otimes C\).
\end{notation}
We can now construct an \(\infty\)-category classifying maps to the unit.
\begin{definition}
	Let \(\mathbb{S}\{v\} \in \calg (\Sp^{\mathbb{N}})\)
	be the free \(\ei\)-algebra on a generator \(\mathbb{S}(1)\).
\end{definition}
\begin{lemma}\label{existence1}
	The symmetric monoidal locally graded
	\(\infty\)-category
	\(\rmod_{\mathbb{S}\{v\}}(\Sp^{\mathbb{N}})
	\in \calg \left(\prgd\right)\)
	has the universal property,
	that for any other
	\(\mathcal{C} \in \calg \left(\prgd\right)\)
	we have
	\begin{align*}
		\fun^{\otimes}_{\prgd}
		\big(\rmod_{\mathbb{S}\{v\}}(\Sp^{\mathbb{N}}),\mathcal{C}\big)
		\simeq
		\map_{\mathcal{C}} (1_{\mathcal{C}}(1),1_{\mathcal{C}}).
	\end{align*}
\end{lemma}
\begin{proof}
	From \cite[Remark 4.8.5.12]{HA},
	along with \(\Sp^{\mathbb{N}}\) being the unit of \(\prgd\),
	we get an adjunction
	\begin{align*}
		\fun^{\otimes}_{\prgd}
		\left(\rmod_B(\Sp^{\mathbb{N}}),\mathcal{C}\right)
		\simeq
		\map_{\calg(\Sp^{\mathbb{N}})}
		\left(B,\End^{\gd}(\mathcal{C})\right),
	\end{align*}
	where \( \End^{\gd} \) is the graded algebra of endomorphisms of the unit
	described in \cite[Remark 2.4.9]{Rotation}.
	For \(\mathbb{S}\{v\}\) we then get
	\begin{align*}
		\fun^{\otimes}_{\prgd}
		\left( \rmod_{\mathbb{S}\{v\}}( \Sp^{\mathbb{N}}), \mathcal{C} \right)
		\simeq
		\map_{\calg( \Sp^{\mathbb{N}}) }\left( \mathbb{S}\{v\}, \End^{\gd}(\mathcal{C}) \right)
		\\
		\simeq
		\map_{\Sp^{\mathbb{N}} }\left( \mathbb{S}(1), \End^{\gd}( \mathcal{C}) \right)
		\simeq
		\map_{ \mathcal{C} }
		\left( \textbf{1}_{ \mathcal{C} }(1), \textbf{1}_{\mathcal{C} }\right).
	\end{align*}
\end{proof}

\begin{remark}
	Explicitly given a map \(\varphi: 1_{\mathcal{C}}(1) \to 1_{\mathcal{C}}\),
	the functor
	\(\rmod_{\mathbb{S}\{v\}}(\Sp^{\mathbb{N}}) \to \mathcal{C}\)
	sends the map \(\cdot v: \mathbb{S}\{v\}(1) \to \mathbb{S}\{v\}\) adjoint
	to \(v:\mathbb{S}(1) \to \mathbb{S}\{v\}\) to \(\varphi\).
	The quotient
	\(\mathbb{S}\{v\}/(\cdot v)\) is then also sent to
	\(1_{\mathcal{C}}/\varphi\).
\end{remark}
It follows that any algebra structure on
\(\mathbb{S}\{v\}/(\cdot v)\)
will induce the same structure
on any cofiber of a strict element.
Asking for algebraic structures in this generality
is too much, as for example \(\mathbb{S}/2\)
does not even admit an \(\an[2]\)-structure,
even though \(\mathbb{S}\) is strict, as it is the unit.
Instead we can modify our universal example,
to get a universal algebra structure in specific cases.

\subsection{\texorpdfstring{\(p\)-local \(\infty\)-categories}{p-local infinity-categories}}
The first modification we consider,
is to replace the category spectra \(\Sp\) by its \(p\)-localization,
which we recall in \Cref{bousfield}.
In this case, we have a similar version of \Cref{existence1}.

\begin{lemma} \label{universalpropertyp}
	The symmetric monoidal locally graded \(p\)-local
	\(\infty\)-category\\
	\(\rmod_{\mathbb{S}_{(p)}\{v\}}(\Sp_{(p)}^{\mathbb{N}})
	\in \calg \left( \prp \right)\)
	has the universal property
	that for any other \(\mathcal{C} \in \calg \left( \prp \right)\),
	we have
	\begin{align*}
		\fun^{\otimes}_{\prp}
		\big(\rmod_{\mathbb{S}_{(p)}\{v\}}(\Sp_{(p)}^{\mathbb{N}})
		, \mathcal{C}\big)
		\simeq
		\map_{\mathcal{C}} (1_{\mathcal{C}}(1),1_{\mathcal{C}}).
	\end{align*}
\end{lemma}
\begin{proof}
	Same as \Cref{existence1},
	in the \(p\)-local setting.
\end{proof}

We now study \(\mathbb{S}_{(p)}\{v\}\) closer,
to see what algebra structures its quotient admits.
By \cite[Example 3.1.3.14]{HA},
the underlying graded spectrum of \(\mathbb{S}_{(p)}\{v\}\) is given by
\begin{align*}
	\bigoplus_{n \geq 0 } \sym_{\Sp^{\mathbb{N}}}^n( \mathbb{S}_{(p)}(1) )
	\simeq
	\bigoplus_{n \geq 0 } \sym_{\Sp}^n ( \mathbb{S}_{(p)} ) (n)
	\simeq
	\bigoplus_{n \geq 0 } \mathbb{S}_{(p)}[\mathrm{B}\Sigma_{n}](n),
\end{align*}
where \(\Sigma_n\) is the symmetric group.
\begin{lemma}\label{labelsarehard}
	The map
	\(\mathbb{S}_{(p)}[\mathrm{B}\Sigma_{n-1}] \to
	\mathbb{S}_{(p)}[\mathrm{B}\Sigma_{n}]\)
	is an equivalence for \(p \nmid n\).
\end{lemma}
\begin{proof}
	Since the spectra are \(p\)-local, it is enough to show that the map
	induces an isomorphism on \(\mathbb{F}_{p}\)-homology,
	which is given by the \(\mathbb{F}_{p}\)-group homology of \(\Sigma_{n}\).
	Since \(i: \Sigma_{n-1} \hookrightarrow \Sigma_{n}\) is an inclusion of index
	\(n\), we have a transfer map
	\begin{align*}
		\tau: H_{n}(\Sigma_{n}, \mathbb{F}_{p}) \to H_{n}(\Sigma_{n-1}, \mathbb{F}_{p})
	\end{align*}
	such that \(H_{n}(i) \circ \tau = n \cdot \id\).
	Since \(n\) is invertible in \(\mathbb{F}_{p}\) this is an isomorphism,
	so \(H_{n}(i)\) is an equivalence.
\end{proof}

\begin{proposition}
	The graded spectrum \(\mathbb{S}_{(p)}\{v\}/(\cdot v)\)
	admits an \(\an[p-1]\)-algebra structure.
\end{proposition}
\begin{proof}
	Applying \Cref{obstruction}
	to \(\mathbb{S}_{(p)}\{v\}/(\cdot v)\),
	we get obstructions for an \(\an[p-1]\)-structure
	\begin{align*}
		\theta_{n} \in \left[\mathbb{S}_{(p)}^{2n-3}(n),
		\mathbb{S}_{(p)}\{v\}/(\cdot v)\right] \; & n = 2 \dots p-1.
	\end{align*}
	The spectrum \(\mathbb{S}_{(p)}\{v\}/(\cdot v)(n)\) in filtration \(n\) is
	given by
	\begin{align*}
		\cofib\left(\mathbb{S}_{(p)}[\mathrm{B}\Sigma_{n-1}] \to \mathbb{S}_{(p)}[\mathrm{B}\Sigma_{n}]\right)
	\end{align*}
	which by \Cref{labelsarehard},
	is an equivalence for \(n < p\),
	so the obstructions \(\theta_{2}, \dots , \theta_{p-1}\) lies in null-groups.
\end{proof}
\begin{corollary}
	For a category
	\(\mathcal{C} \in \calg \left( \prp \right)\),
	and a map \(\varphi: 1_{\mathcal{C}}(1) \to 1_{\mathcal{C}}\),
	the cofiber \(1_{\mathcal{C}}/\varphi\) admits an \(\an[p-1]\)-algebra structure.
	In particular given a \(p\)-local \(\infty\)-category \(\mathcal{C} \in
	\calg(\prst_{(p)})\) and a map from a strict element to the unit
	\(\varphi: X \to 1_{\mathcal{C}}\) the cofiber \(1_{\mathcal{C}}/\varphi\)
	admits an \(\an[p-1]\)-algebra structure.
\end{corollary}
\begin{example}
	Since \(\mathbb{S}/p\) is the quotient
	\(\mathbb{S}_{(p)}/p\mathbb{S}_{(p)}\),
	and \(\mathbb{S}_{(p)}\)
	is strict since it is the unit in
	\(\Sp_{(p)}\), we get that it admits an \(\an[p-1]\)-algebra structure.
\end{example}

\subsection{\texorpdfstring{2-local \(\infty\)-Categories}{2-local infinity-Categories}}
We now specialize to \(2\)-local \(\infty\)-categories and show that if the
\(\an[2]\)-obstruction vanishes,
then the \(\an[3]\)-obstruction will also vanish.
\begin{definition}
	Let
	\(Q_{1}(v): \mathbb{S}_{(2)}^{1}(2) \to \mathbb{S}_{(2)}\{v\}\)
	be the power operation described in \cite[Lemma 5.4]{Burklund}
	and let
	\(\cdot Q_{1}(v): \mathbb{S}^1_{(2)}\{v\}(2) \to \mathbb{S}_{(2)}\{v\}\)
	be the adjoint map in \(\rmod_{\mathbb{S}_{(2)}\{v\}}(\Sp^{\mathbb{N}})\).
	Given a category
	\( \mathcal{C} \in \calg(\prtwo)\)
	and a map
	\(\varphi: \mathbf{1}_{\mathcal{C}}(1)\to  \mathbf{1}_{\mathcal{C}}\),
	define
	\(Q_{1}(\varphi): \Sigma \mathbf{1}_{\mathcal{C}}(2)\to \mathbf{1}_{\mathcal{C}}\)
	as the image of \(\cdot Q_{1}(v)\) under the functor given by \(\varphi\) by
	\Cref{universalpropertyp}.
\end{definition}

\begin{definition}
	Let
	\(\mathbb{S}_{(2)} \{ v\}, \mathbb{S}_{(2)} \{ w\} \in \calg(\Sp^{\mathbb{N}}) \)
	be the free \( \ei \)-algebras on generators \( v,w \),
	where \( v \) has degree
	0 and filtration 1,
	and \( w \) has degree 1 and filtration 2.
	We define \( A \) as the pushout in
	\( \calg\left( \Sp_{(2)}^{\mathbb{N}}\right) \) of the
	diagram
	\[\begin{tikzcd}
			{\mathbb{S}_{(2)} \{ w\}} & { \mathbb{S}_{(2)} \{ v \}} \\
			{\mathbb{S}_{(2)} } & {A.}
			\arrow["{Q_1(v)\cdot}", from=1-1, to=1-2]
			\arrow["0\cdot"', from=1-1, to=2-1]
			\arrow[from=1-2, to=2-2]
			\arrow[from=2-1, to=2-2]
			\arrow["\ulcorner"{anchor=center, pos=0.125, rotate=180}, draw=none, from=2-2, to=1-1]
		\end{tikzcd}\]
\end{definition}
\begin{lemma}\label{existence}
	The symmetric monoidal stable locally graded 2-local presentable
	\(\infty\)-category
	\(\rmod_{A}(\Sp^{\mathbb{N}}) \in \calg \left( \prtwo \right)\),
	has the universal property that for any
	\(\mathcal{C} \in \calg \left( \prtwo \right)\),
	the diagram
	\[\begin{tikzcd}
			{\fun^{\otimes}_{\prtwo}
			(\rmod_{A}(\Sp^{\mathbb{N}}),\mathcal{C})}
			&
			{*} \\
			{\map_ { \mathcal{C} }( 1_ \mathcal{C}(1),
			1_ { \mathcal{C} })} & {\map_{ \mathcal{C} }(\Sigma 1_ \mathcal{C}(2),
					1_ { \mathcal{C} })}
			\arrow[from=1-1, to=1-2]
			\arrow[from=1-1, to=2-1]
			\arrow["\ulcorner"{anchor=center, pos=0.125}, draw=none, from=1-1, to=2-2]
			\arrow["{0}", from=1-2, to=2-2]
			\arrow["{Q_1}"', from=2-1, to=2-2]
		\end{tikzcd}\]
	is a pullback square,
	where the lower map assigns \( Q_1 \) to the map
	\( 1_{\mathcal{C}}(1) \to 1_ {\mathcal{C}} \).
\end{lemma}
\begin{proof}

	From \cite[Remark 4.8.5.12]{HA},
	we have
	\begin{align*}
		\fun^{\otimes}_{\prgd}
		\left( \rmod_A( \Sp^{\mathbb{N}}), \mathcal{C} \right) \\
		\simeq
		\map_{\Sp^{\mathbb{N}} }\left( \mathbb{S}_{(2)}(1), \End^{\gd}( \mathcal{C}) \right)
		\times_{\map_{\Sp^{\mathbb{N}}}\left(\mathbb{S}_{(2)}^1(2), \End^{\gd}( \mathcal{C} )
		\right)} *                                             \\
		\simeq
		\map_{ \mathcal{C} }( \textbf{1}_{ \mathcal{C} }(1), \textbf{1}_{
		\mathcal{C} })
		\times_
		{
		\map_{ \mathcal{C} }(\Sigma \textbf{1}_{ \mathcal{C} }(2), \textbf{1}_{
		\mathcal{C} })
		}*.
	\end{align*}
	A symmetric monoidal functor in \( \prtwo \) out of
	\( \rmod_{A}( \Sp^{\mathbb{N}}) \),
	is then given by a
	map \(\varphi: \textbf{1}_{ \mathcal{C} }(1) \to \textbf{1}_{ \mathcal{C} }
	\),
	along with a nulhomotopy of \(Q_1( \varphi )\).
\end{proof}
\begin{proposition}\label{cofib}
	The cofiber \(A /v\) admits an \( \mathbb{A}_3 \)-multiplication.
\end{proposition}
\begin{corollary}\label{coolthing}
	Let \(X\) be a strict element in \( \mathcal{C} \in \calg(\prst_{(2)})\),
	and \(\varphi: X \to \mathbf{1}_{\mathcal{C}}\) be a map to the unit.
	If the map
	\( Q_1(\varphi): \Sigma X^{\otimes 2}
	\to
	\textbf{1}_{\mathcal{C}} \)
	vanishes,
	then \( 1_{ \mathcal{C} }/ \varphi  \)
	admits an \(\mathbb{A}_3\)-multiplication.
\end{corollary}
To prove \cref{cofib} ,
we will analyze the map \(A(1) \xto{\cdot v} A\) to show the obstructions
land in null-groups.

\begin{lemma}\label{alem}
	The map
	\( A(1)
	\xrightarrow{\cdot v}
	A\)
	induce an isomorphism in filtration 3.
\end{lemma}
\begin{proof}
	Since pushouts of \(\ei\)-algebras are given by relative tensor
	products, we have
	\begin{align*}
		A \simeq \mathbb{S}_{(2)}\{v\} \otimes_{\mathbb{S}_{(2)}\{w\}}
		\mathbb{S}_{(2)}.
	\end{align*}
	Further, we see from filtration 0 to 3, \(\mathbb{S}_{(2)}\{w\}\)
	has only two cells 1 and \(w\),
	matching the free \(\eone\)-algebra \(\mathbb{S}_{(2)}[w]\),
	so from filtration 0 to 3 they agree.
	We can also write \(\mathbb{S}_{(2)}\) as a quotient
	\begin{align*}
		\mathbb{S}_{(2)} =
		\cofib \left( \Sigma\mathbb{S}_{(2)}[w](2) \xto{\cdot w}.
		\mathbb{S}_{(2)}[w] \right)
	\end{align*}
	Using this along with tensor products commuting with cofibers,
	we get a cofiber sequence
	\begin{align*}
		\Sigma\mathbb{S}_{(2)}\{v\}(2) \xto{\cdot Q_{1}(v)}
		\mathbb{S}_{(2)}\{v\} \to
		A.
	\end{align*}
	We then have the map \(A(1) \xto{\cdot v} A \) is given by the map of
	cofibers
	\[\begin{tikzcd}
			{\Sigma \mathbb{S}_{(2)}(3)} & {\mathbb{S}_{(2)}(1)} & {A(1)} \\
			{\Sigma \mathbb{S}_{(2)}(2)} & {\mathbb{S}_{(2)}(2)} & A,
			\arrow["{\cdot Q_1(v)}", from=1-1, to=1-2]
			\arrow[from=1-1, to=2-1]
			\arrow[from=1-2, to=1-3]
			\arrow[from=1-2, to=2-2]
			\arrow[from=1-3, to=2-3]
			\arrow["{\cdot Q_1(v)}", from=2-1, to=2-2]
			\arrow[from=2-2, to=2-3]
		\end{tikzcd}\]
	which in filtration 3 is
	\[\begin{tikzcd}
			{\Sigma \mathbb{S}_{(2)}[\Sigma_{0}]} & {\mathbb{S}_{(2)}[\Sigma_{2}]} & {A(1)} \\
			{\Sigma \mathbb{S}_{(2)}[\Sigma_{1}]} & {\mathbb{S}_{(2)}[\Sigma_{3}]} & A
			\arrow["{\cdot Q_1(v)}", from=1-1, to=1-2]
			\arrow["\simeq", from=1-1, to=2-1]
			\arrow[from=1-2, to=1-3]
			\arrow["\simeq", from=1-2, to=2-2]
			\arrow[from=1-3, to=2-3]
			\arrow["{\cdot Q_1(v)}", from=2-1, to=2-2]
			\arrow[from=2-2, to=2-3]
		\end{tikzcd}\]
	where the vertical maps are equivalences by \cref{labelsarehard}.

\end{proof}

\begin{sseqdata}[ name = Freegradedw, xscale=2, yscale=2, x range = {0}{4}, y range
			= {0}{4}, x tick step = 2, y tick step = 2, axes type = frame, class labels = {left}, classes = fill, grid = crossword, Adams grading, lax degree]
	\class(0,0) \node[above] at (0,0) {1};

	\class(1,0) \node[above] at (1,0) {v};
	\foreach \x in {2,...,3} {
			\class(\x,0) \node[above] at (\x,0) {v^{\x}};
		}
	\foreach \x in {1,...,5} {
			\class(2,\x) \node[left] at (2,\x) {Q_{\x}(v)};
		}

	\foreach \x in {1,...,4} {
			\class(3,\x) \node[left,below] at (3,\x) {Q_{\x}(v)\cdot v };
		}
	\class(2,2) \structline(2,1);
	\structline(2,2)(2,1);
	\structline(2,4)(2,3);
	\node[above] at (2.1,2) {w};

	\class(3,2) \structline(3,1);
	\structline(3,2)(3,1);
	\structline(3,4)(3,3);
	\node[above] at (3.4,1.85) {w\cdot v};
	\structline[bend right = 12](3,4)(3,2);
	\structline[bend right = 12](2,4)(2,2);
\end{sseqdata}
Since \(\mathbb{S}_{(2)}[\mathrm{B}\Sigma_{2}]\) is 2-local and of finite type,
we can find a minimal cellular structure with number of \(i\)-cells given by
\begin{align*}
	\dim_{\mathbb{F}_{2}}(H_{i}(\mathbb{S}_{(2)}[\mathrm{B}\Sigma_{2}], \mathbb{F}_{2}))\simeq
	\dim_{\mathbb{F}_{2}}(H_{i}(\mathbb{R}P^{\infty}, \mathbb{F}_{2}))\simeq
	1.
\end{align*}
Along with the cofiber sequence
\begin{align*}
	\Sigma\mathbb{S}_{(2)}\{v\}(2) \xto{\cdot Q_{1}(v)}
	\mathbb{S}_{(2)}\{v\} \to
	A
\end{align*}
We then get the following cellular structure for \(A\) displayed in
\cref{syn4}.
\begin{figure}[ht!]
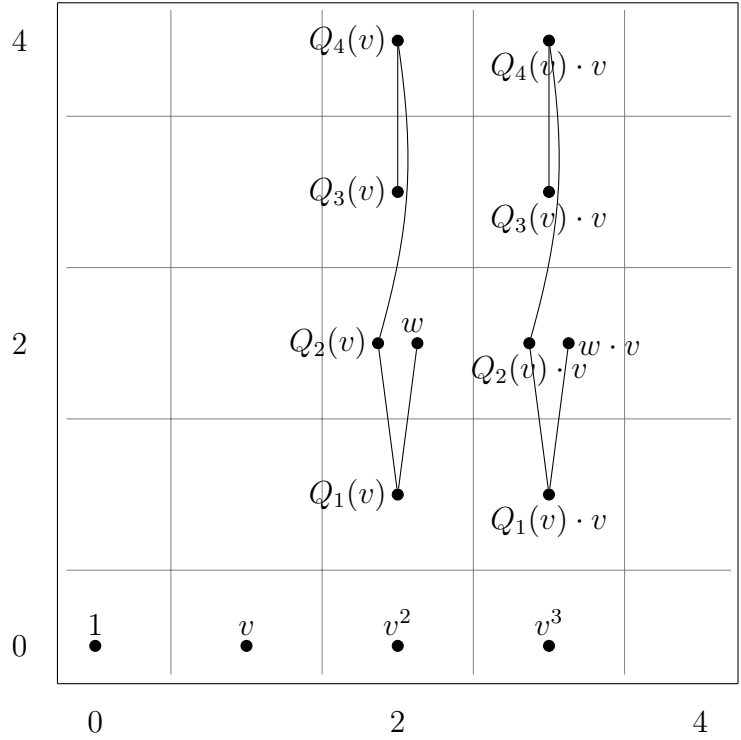

	\centering
	\printpage[ name = Freegradedw, page = 2 ]
	\caption{Cell Structure of \( A \),
		in filtration 0 to 3,
		with horizontal axis giving filtration and vertical axis giving the degree.
		Each dot represent a copy of the sphere spectrum.}\label{syn4}
\end{figure}
\begin{sseqdata}[ name = Freegraded, xscale=2, yscale=2, x range = {0}{3}, y range
			= {0}{4}, x tick step = 2, y tick step = 2, axes type = frame, class labels = {left}, classes = fill, grid = crossword, Adams grading, lax degree]
	\class(0,0) \node[above] at (0,0) {1};

	\class(1,0) \node[above] at (1,0) {v};
	\foreach \x in {2,...,3} {
			\class(\x,0) \node[above] at (\x,0) {v^{\x}};
		}
	\foreach \x in {1,...,5} {
			\class(2,\x) \node[left] at (2,\x) {Q_{\x}(v)};
		}

	\foreach \x in {1,...,4} {
			\class(3,\x) \node[left] at (3,\x) {Q_{\x}(v)\cdot v };
		}
	\structline(2,2)(2,1);
	\structline(2,4)(2,3);
	\node[above] at (1.75,1.4) {\cdot 2};
	\node[below] at (2,3) {\cdot \eta};
	\node[above] at (1.8,3.3) {\cdot 2};

	\structline(3,2)(3,1);
	\structline(3,4)(3,3);
	\node[above] at (2.75,1.4) {\cdot 2};
	\node[above] at (2.8,3.3) {\cdot 2};
	\node[below] at (3,3) {\cdot \eta};
	\structline[bend right = 12](3,4)(3,2);
	\structline[bend right = 12](2,4)(2,2);
\end{sseqdata}


\begin{proof}[Proof of \Cref{cofib}]
	In filtration 2,
	the cell \(Q_{1}(v)\) is killed by the attaching map in the definition of \(A\),
	and the cell \(v^{2}\) is killed by the map \(A(1) \xto{\cdot v} A\),
	so there are no maps from spheres of degree less than \( 2 \).
	The group
	\begin{align*}
		[\mathbb{S}_{(2)}^1(2), A/v]_{\Sp^{\mathbb{N}}},
	\end{align*}
	which contains the \(\an[2]\)-obstruction
	is therefore null.
	From \Cref{alem},
	we know that \(A/v\) vanishes in filtration 3,
	so the group
	\begin{align*}
		[\mathbb{S}_{(2)}^3(3), A/v]_{\Sp^{\mathbb{N}}}.
	\end{align*}
	which contains the \(\an[3]\)-obstruction is also null.
	We can then conclude from \Cref{obstruction},
	that \(A/v\) admits a \(\mathbb{A}_3\)-multiplication.
\end{proof}

%
%
%


\section{Relating Obstructions in Different Categories}
In this section we show that given a symmetric monoidal stable \(\infty\)-category
\(\mathcal{C}\), if \(A \in \map_{\ei}(\mathbb{N}, \mathcal{C})\)
is a strict element,
then the obstruction theory given in \cref{obstruction},
factors through the unit map \(1_{\mathcal{C}} \to A\).

\subsection{Map between Obstructions}\label{obstructionmap}
A monoidal functor \(F\)
induces a functor on \(\infty\)-categories of \(\an\)-algebras.
The following proposition shows that \(F\) also maps the obstruction from
\Cref{obstruction} to the obstruction on the target.

\begin{proposition}\label{obstructionmapp}
	Let
	\(F: \mathcal{C}^{\otimes} \to \mathcal{D}^{\otimes}\)
	be an exact monoidal functor of stable monoidal \(\infty\)-categories,
	\( A \in \alg_{\mathbb{A}_{n-1}}(\mathcal{C})\)
	be an \(\mathbb{A}_{n-1}\)-algebra with unit
	\(u: 1_{\mathcal{C}} \to A\),
	and \(\theta_{n}\) be the obstruction theory from \Cref{obstruction}.

	The functor \(F\) then sends \(\theta_n\)
	for \(A\)
	to the obstruction \(\theta_{n}\) for \(F(A)\)
	by the map
	\begin{align*}
		\map_{\mathcal{C}}(\Sigma^{n - 3} (A/u) ^{\otimes n},A)
		\xrightarrow{F}
		\map_{\mathcal{D}}(\Sigma^{n - 3} (F(A)/F(u))^{\otimes n},F(A)).
	\end{align*}
	If \(F\) admits a right adjoint \(R\),
	the above map factorises as
	\begin{align*}
		\map_{\mathcal{C}}(
		\Sigma^{2 n - 3} A/u ^{ \otimes n },
		A
		)
		\xrightarrow{\eta _ { A } \circ -}
		\map_{ \mathcal{C} }
		(
		\Sigma ^ {2 n - 3} A/u ^ { \otimes n },
		RF ( A )
		)
		\\
		\xrightarrow{ \Psi }
		\map_{ \mathcal{D} }(
		\Sigma^{2 n - 3} F(A)/F(u) ^{ \otimes n },
		F(A)
		).
	\end{align*}

	Where \(\Psi\) is the adjunction equivalence,
	and \(\eta\) is the unit of the adjunction.

\end{proposition}
\begin{proof}
	Since \(F\) is exact, we get a map of cofiber sequences
	\begin{figure}[H]
		\begin{adjustbox}{width=\textwidth}
			\begin{tikzcd}
				{\map_{ \mathcal{ C } }
					\left( A ^ { \otimes n } , A \right) ^ { K_n }}
				&
				{\map_{ \mathcal{ D } }
						\left( F(A) ^ { \otimes n } , F(A) \right) ^ { \partial K_n }}
				\\
				{\map_{ \mathcal{ C } }
				\left(  A ^ { \otimes n } , A \right) ^ { \partial K_n }
				\times_{\map_{ \mathcal{C}}
					\left(T, A\right)^{\partial K_n}}
				\map_{ \mathcal{C}}\left(T,A\right)^{ K_n}}
				&
				{\map_{ \mathcal{ D } }
						\left( F(A) ^ { \otimes n } , F(A) \right) ^ { \partial K_n }
						\times_{\map_{ \mathcal{D}}\left(T, F(A) \right)^{\partial K_n}}
						\map_{ \mathcal{D}}\left(T, F(A) \right)^{ K_n}}
				\\
				{\map_{\mathcal{C}}
				\left( \Sigma ^ {  n - 3} A/u ^ { \otimes n } , A \right)}
				&
				{\map_{\mathcal{D}}( \Sigma ^ {  n - 3} F(A)/F(u) ^ { \otimes n } , F(A)).}
				\arrow["F",from=1-1, to=1-2]
				\arrow[from=1-1, to=2-1]
				\arrow[from=1-2, to=2-2]
				\arrow["F",from=2-1, to=2-2]
				\arrow[from=2-1, to=3-1]
				\arrow[from=2-2, to=3-2]
				\arrow[dashed, from=3-1, to=3-2]
			\end{tikzcd}
		\end{adjustbox}
	\end{figure}
	Since the two first horizontal maps are given by \( F \),
	the induced map is also given by \( F \).
	The statement of right adjoints, follows from the definition of an
	adjunction.
\end{proof}

\begin{example}\label{module}
	Let \( A,B \) be \( \mathbb{E}_2 \)-algebras in a symmetric monoidal category
	\( \mathcal{C}^{\otimes} \),
	and let \( \varphi: A \to B \)
	be a morphism in
	\( \alg_{\mathbb{E}_2}( \mathcal{C} ) \).
	Then the functor
	\begin{align*}
		\rmod_A( \mathcal{C} )
		  &
		\to
		\rmod_B( \mathcal{C} )    \\
		M & \mapsto M \otimes_A B
	\end{align*}
	is monoidal by
	\cite[Theorem 4.8.5.16]{HA},
	with right adjoint given by restriction of scalars.
\end{example}

\subsection{
	\texorpdfstring{\(\mathbb{E}_{\infty }\)}{Einf}-Ring structures on free
	\texorpdfstring{\(\mathbb{E}_{1}\)}{eone}-rings}
We will now construct a map of \( \ei \)-rings to use
\Cref{obstructionmapp} and \Cref{module}.
To do so we use the Thom construction given in \cite{thom}.
While \cite{thom} only deals with groups,
the construction only uses the underlying monoid structure,
and so works equally well for monoids.

\begin{definition}
	Let \(\mathcal{C}\) be a presentably symmetric monoidal \(\infty\)-category,
	and \(M\) an \(\ei\)-monoid.
	The Thom functor is then given by
	\begin{align*}
		\Th_{\mathcal{C}}:
		\fun_{\ei}(M,\mathcal{C})
		\to & \alg_{\ei}(\mathcal{C}) \\
		\left(\varphi: M \mapsto \mathcal{C}\right)
		\to
		    &
		\colim_M \left(\varphi \right),
	\end{align*}
	where the monoidal structure is constructed in \cite{thom}.
\end{definition}
The following proposition gives a unique characterization of
\( \Th_{\mathcal{C}} \) up to isomorphism.

\begin{proposition}[\cite{thom}]\label{unique-thom}
	Let \(M\)  be an \(\mathbb{E}_n\)-monoid
	and suppose that for every \( \mathcal{C} \in \calg(\prl) \),
	we have a functor
	\begin{align*}
		\Th'_{ \mathcal{C} } : \map_{\mathbb{E}_n}( M ,\mathcal{C} )
		\to
		\alg_{ \mathbb{E}_n} ( \mathcal{C} )
	\end{align*}
	lifting \( \colim_M \) along
	\( \alg_n ( \mathcal{C} ) \to \mathcal{C}\),
	and such that for every
	\(F: \mathcal{C} \to \mathcal{D} \in \alg_{\mathbb{E}_n}(\prl) \),
	we have a natural isomorphism
	\begin{align*}
		F(\Th_{ \mathcal{C} }( \xi )) \simeq \Th_{ \mathcal{D} } ( F ( \xi )).
	\end{align*}
	Then for every \(\mathcal{C}\), we have an isomorphism of functors
	\( \Th'_{\mathcal{C}} \simeq \Th_{\mathcal{C}} \).
\end{proposition}

\begin{corollary}
	The compositions
	\begin{align*}
		\map_{\ei} (\mathbb{N} ,\mathcal{C})
		\xrightarrow{\Th_{\mathcal{C}}}
		\alg_{\ei} (\mathcal{C})
		\xrightarrow{\mathrm{res}}
		\alg_{\mathbb{E}_{1}} (\mathcal{C})
	\end{align*}
	and
	\begin{align*}
		\map_{\ei} (\mathbb{N},\mathcal{C})
		\xrightarrow{\ev_{\langle 1 \rangle}}
		\mathcal{C}
		\xrightarrow{ \mathrm{Free}_{ \mathbb{E} _{1}} }
		\alg_{ \mathbb{E}_{1} } ( \mathcal{C} )
	\end{align*}
	are isomorphic.
\end{corollary}
\begin{proof}
	The functor \( \mathrm{Free}_{\mathbb{E}_1}\) exists,
	and its underlying object agrees with
	\( \colim_{\mathbb{N}} \varphi\) in \(\mathcal{C}\) by
	\cite[Proposition 4.1.1.18]{HA},
	so both compositions are lift of \(\colim_{\mathbb{N}} \varphi\).
	Furthermore given a functor \(F\) in \(\alg_{\ei}(\prl)\)
	with right
	adjoint \(R\),
	we have a commutative square
	\[\begin{tikzcd}
			{\alg_{ \mathbb{E}_1}( \mathcal{C})} & {\alg_{ \mathbb{E}_1}( \mathcal{D})} \\
			{\mathcal{C}} & {\mathcal{D},}
			\arrow["R", from=1-2, to=1-1]
			\arrow["\theta"', from=1-1, to=2-1]
			\arrow["\theta", from=1-2, to=2-2]
			\arrow["R", from=2-2, to=2-1]
		\end{tikzcd}\]
	so
	\( \mathrm{Free}_{ \mathbb{E}_1} \circ F \simeq
	F \circ \mathrm{Free}_{ \mathbb{E}_1}\),
	since they are left adjoints to isomorphic functors.
	It follows that both
	\( \mathrm{res} \circ\Th_{ \mathcal{C} }\)
	and
	\( \mathrm{Free}_{\mathbb{E}_1} \circ \ev_{\langle 1 \rangle} \) uphold the properties of \Cref{unique-thom}
	and are therefore isomorphic.
\end{proof}
It follows that the free \(\mathbb{E}_{1}\)-algebra on a strict element
admits an \(\mathbb{E}_{\infty}\)-algebra structure,
given by the Thom functor.
\begin{remark}
	Let \(\mathcal{C}\) be a presentably symmetric monoidal pointed
	\(\infty\)-category,
	so it has an object \(*\) which is both initial and final.
	Then \(*\) is a strict element,
	which is final in
	\(\fun_{\ei}(\mathbb{N}, \mathcal{C})\).
	We then get for any strict element \(X\),
	an \(\ei\)-map \(\Th(X) \to \Th(*)\simeq 1_{\mathcal{C}}\).
\end{remark}

\subsection{Obstructions on cofiber of invertible Elements}
We will give a condition for the obstruction to factor through the unit.
\begin{construction}
	Let \(\mathcal{C}\) be a presentably symmetric monoidal stable \(\infty\)-category
	and \(\varphi: X \to 1_{\mathcal{C}}\)
	be a map from a strict element to the unit of \(\mathcal{C}\).
	Since the underlying object of
	\(\Th(X)\) is \(\oplus_{\mathbb{N}} X^{\otimes n}\),
	we get a map
	\(x-\varphi: X \to \Th(X)\), where \(x\) denotes the generator of \(\Th(X)\).
	Taking the adjoint of this map
	under the free/forgetful adjunction
	\[\begin{tikzcd}
			{\mathcal{C}} & {\rmod_{\Th({X})}(\mathcal{C})}
			\arrow["{\mathrm{Free}}", shift left, from=1-1, to=1-2]
			\arrow["{\mathrm{Forgetful}}", shift left, from=1-2, to=1-1]
		\end{tikzcd}\]
	we get a map
	\(\cdot (x-\varphi):\Th(X)[X] \to \Th(X)\).
	We denote the cofiber by \(\fourmodule\).
\end{construction}
\begin{remark}
	The notation is supposed to suggest that \( \fourmodule \)
	has underlying object \(1_{\mathcal{C}}\),
	with \(\Th(X)\) acting by \(\varphi\).
\end{remark}
The \(\ei\)-ring map \(\Th(X) \to 1_{\mathcal{C}}\)
induces a symmetric monoidal functor on module categories
by the pushforward functor
\[\begin{tikzcd}
		{\rmod_{\Th({X})}(\mathcal{C})} & {\rmod_{1_{\mathcal{C}}}(\mathcal{C})} & {\mathcal{C}.}
		\arrow[from=1-1, to=1-2]
		\arrow["\simeq", from=1-2, to=1-3]
	\end{tikzcd}\]
\begin{lemma}
	The symmetric monoidal functor
	\( \rmod_{\Th(X)}(\mathcal{C}) \to \mathcal{C} \)
	sends \(\fourmodule\) to \(1_{\mathcal{C}}/\varphi\).
\end{lemma}
\begin{proof}
	The pushforward is a left adjoint,
	so it commutes with colimits.
	We then have
	\begin{align*}
		\cofib_{\Th(X)}(\Th(X)[X]
		\xrightarrow{\cdot(x-\varphi)}
		\Th(X))
		\otimes_{\Th(X)}
		1_{\mathcal{C}}
		\simeq
		\cofib_{\mathcal{C}}( X \xto{\varphi}	1_{\mathcal{C}})
		\simeq 1_{\mathcal{C}}/\varphi.
	\end{align*}
\end{proof}
Combining this lemma with \Cref{obstructionmapp} and \Cref{module},
we get
\begin{proposition}\label{factor}
	Let \(\mathcal{C}\) be a presentably symmetric monoidal stable \(\infty\)-category,
	let \(X\) be a strict element and \(X \to 1_{\mathcal{C}}\) a map to the
	unit.
	An \(\mathbb{A}_{n-1}\)-algebra structure on
	\(\fourmodule\) over \(\Th(X)\),
	induces an \(\mathbb{A}_{n-1}\)-algebra structure on
	\(1_{\mathcal{C}}/\varphi\),
	and the obstruction for an \(\an\)-algebra structure on
	\(\fourmodule\) over \(\Th(X)\),
	gets mapped to the obstruction for an \(\an\)-algebra on
	\(1_{\mathcal{C}}/\varphi\)
	by postcomposition with the cofiber map
	\(1_{\mathcal{C}} \to 1_{\mathcal{C}}/\varphi\)
	\begin{align*}
		[\Sigma^{2n-3}1_{\mathcal{C}}, 1_{\mathcal{C}}]_{\mathcal{C}}
			=
			[\Sigma^{2n-3} \Th(X),\fourmodule]_{\Th(X)}
		\\
		\xrightarrow{\eta}
		[\Sigma^{2n-3,3} \Th(X),1_{\mathcal{C}}/\varphi]_{\Th(X)}
		\xrightarrow{\Psi}
		[\Sigma^{2n-3} 1_{\mathcal{C}},1_{\mathcal{C}}/\varphi]_{\mathcal{C}}
	\end{align*}
\end{proposition}

\section{Associative structures on
  \texorpdfstring{\(\mathbb{S}/4\)}{s/4}}\label{synthetic}
We will apply the techniques developed so far
to show that \(\mathbb{S}/4\)
admits an \( \mathbb{A}_5 \)-algebra structure.
From \cite{Prasit},
we have that \(\mathbb{S} /4\) already admits an
\( \mathbb{A}_4 \)-algebra structure,
so the relevant obstruction is \( \theta_5 \in \pi_7( \mathbb{S} / 4) \neq 0 \).
This obstruction is hard to calculate concretely,
and since \(\theta_5\) does not lie in a null-group,
it does not vanish automatically.
We will instead apply \Cref{obstruction}
to the \(\infty\)-category of \(\mathbb{F}_{2}\)-synthetic spectra \(\syntwo\).
We refer to \cite{Piotr} and \cite{Boundaries} for more details on this
\(\infty\)-category.

\begin{proposition}
	The synthetic spectrum \(\nu \mathbb{S} /\tilde{4}\)
	admits an \(\mathbb{A}_2\)-multiplication.
\end{proposition}
\begin{proof}
	Applying the obstruction theory from
	\Cref{obstruction} to the map
	\(\tilde{4}: \nu\mathbb{S}^{0,2} \to \nu \mathbb{S}\)
	in \(\syntwo\),
	we get obstructions
	\begin{align*}
		\theta_k \in [ \nu\mathbb{S}^{2k-3,3} , \nu\mathbb{S}/\tilde{4}] =
		\pi_{2k-3,3}(\nu\mathbb{S}/\tilde{4}) \quad k \geq 2.
	\end{align*}
	From  the calculation of the \(\mathrm{E}_2
	\)-page of the Adams spectral sequence in \cite{Chart},
	we get that there are no differentials in topological degree less than or equal
	to 12. \cite[Theorem A.8]{Boundaries} then implies that in this range the
	homotopy groups \( \pi_{t,s}( \nu \mathbb{S} ) \) are given by the \(
	\mathrm{E}_2 \)-page of the \( \mathbb{F}_2 \)-Adams spectral sequence of \( \mathbb{S} \)
	tensored with \( \mathbb{Z}[\tau] \), pictured below without
	\(\tau\)-multiples.

	\begin{sseqdata}[ name = ASSC4, xscale=0.8, yscale=0.8, x range = {0}{8}, y range
				= {0}{8}, x tick step = 2, y tick step = 2, axes type = frame, class labels = {left}, classes = fill, grid = crossword, Adams grading, lax degree]
		\class(0,0)
		\class(0,1)
		\structline

		\class(1,1) \structline(0,0)
		\class(2,2) \structline
		\class(2,2) \structline(1,1)
		\class(3,1) \structline(0,0)
		\class(3,2) \structline \structline(0,1)
		\class(3,3) \structline(2,2)\structline(2,2,2)
		\class(4,3)
		\class(4,4) \structline \structline(3,3)

		\class(6,2) \structline(3,1)
		\class(7,1)
		\class(7,2) \structline
		\class(7,3)
		\class(8,2) \structline(7,1)

		\class(8,3)

		\class(8,4)
		\class(8,5) \structline

		\foreach \x in {2,...,5} {
				\draw[color=red] (2*\x-3,3) circle [radius=0.16];
				\node[color=red,below] at (2*\x-3,3) {\theta_{\x}};
			}
	\end{sseqdata}

	\begin{figure}[ht!]
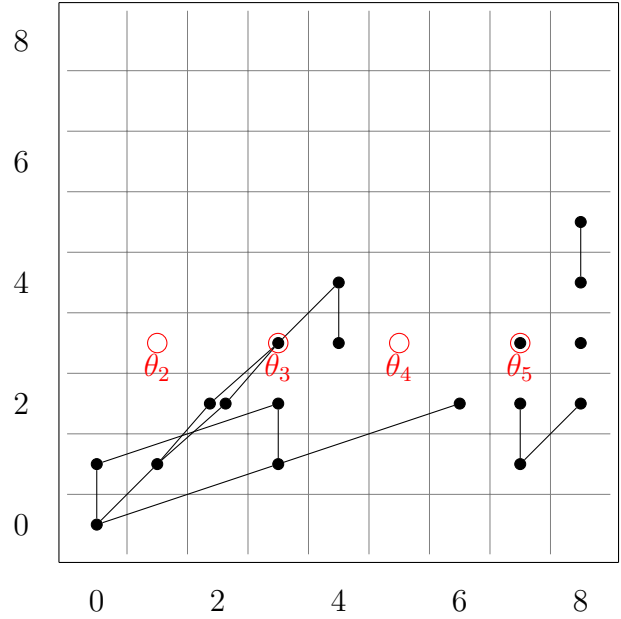

		\centering
		\printpage[ name = ASSC4, page = 2 ]
		\caption{Bigraded homotopy groups of \( \nu \mathbb{S} / \tilde{4} \)
			without \( \tau \)-multiples.
			Each dot represents a copy of \( \mathbb{F}_2 \),}
		\label{syn}
	\end{figure}
	Since \( \tau \in \pi_{0,-1}( \nu\mathbb{S} )  \),
	any \( \tau \)-multiple lives below non \( \tau \)-multiples.
	We then see from \Cref{syn}, that the obstruction
	\( \theta_2 \) vanishes since it lives in a null-group.
\end{proof}

The third obstruction and the fifth obstruction
do not lie in null-groups.
To remedy this, we will show the obstructions factor through the map
\(\nu\mathbb{S} \to \nu\mathbb{S}/\tilde{4}\) by \Cref{factor}.
We first show	the synthetic spectrum \(\nu\mathbb{S}^{0,2}\) is a strict element.
We do this by comparing strict units in \(\Sp\) and \(\syntwo\)
in the following lemma.
\begin{lemma}
	The synthetic spectrum \(\nu \mathbb{S}^{0,2}\) is a strict element.
\end{lemma}
\begin{proof}
	We have a symmetric monoidal splitting
	\begin{align*}
		\Sp
		\xrightarrow{ \nu }
		\syntwo
		\xrightarrow{ \tauinv}
		\Sp,
	\end{align*}
	so we get a splitting of spectra
	\[\begin{tikzcd}
			{\pic(\Sp)}
			&
			{\pic(\syntwo)}
			&
			{\cofib(\pic(\nu)),}
			\arrow["{\pic (\nu)}", from=1-1, to=1-2]
			\arrow["{\pic(\tauinv)}", curve={height=-6pt}, from=1-2, to=1-1]
			\arrow[from=1-2, to=1-3]
		\end{tikzcd}\]
	which induces an equivalence
	\begin{align*}
		\pic(\syntwo)\simeq \pic(\Sp) \oplus \cofib(\pic(\nu)).
	\end{align*}
	The functor \(\nu\) is fully faithful,
	so \(\pi_n \pic(\nu)\) is an isomorphism for \(n \neq 0\)
	and \(\pi_0 \pic(\nu)\) is an injection,
	which causes \(\pi _n \cofib(\pic(\nu))\) to vanish for \(n \neq 0\).
	Since \(\tauinv(\mathbb{S}^{0,2})\simeq \mathbb{S}\),
	so we have
	\begin{align*}
		\nu\mathbb{S}^{0,2}\in \pi_{0}(\cofib(\pic(\nu))).
	\end{align*}
	Since
	\(\cofib(\pic(\nu))\)
	is discrete the map factors through
	\(\mathbb{S} \to \mathbb{Z}\),
	implying \(\mathbb{S}^{0,2}\) is a strict element.
\end{proof}
We now show the second criterion of \Cref{factor}.

\begin{lemma}
	The synthetic spectrum \(\sfourmodule\) is an \(\an[4]\)-algebra
	in \(\rmod_{\Th(\mathbb{S}^{0,2})}(\syntwo)\).
\end{lemma}
\begin{proof}
	The obstructions lie in
	\begin{align*}
		\theta_{n} \in
		[\Sigma^{2n-3}(\Th(\nu\mathbb{S}^{0,2})[\nu \mathbb{S}^{0,2}])^{\otimes n},
		\sfourmodule]_{\Th(\nu\mathbb{S}^{0,2})}
			=
			[\Sigma^{2n-3,3-2n}(\nu\mathbb{S}^{0,2})^{\otimes n},
				\nu \mathbb{S}]
		=
		\pi_{2n-3,3}(\nu \mathbb{S})
	\end{align*}

	\begin{sseqdata}[ name = ASS, xscale=0.8, yscale=0.8, x range = {0}{8}, y range
				= {0}{8}, x tick step = 2, y tick step = 2, axes type = frame, class labels = {left}, classes = fill, grid = crossword, Adams grading, lax degree]
		\class(0,0)
		\foreach \y in {1,...,8} {
				\class(0,\y)
				\structline
			}

		\class(1,1) \structline(0,0)
		\class(2,2) \structline
		\class(3,1) \structline(0,0)
		\class(3,2) \structline \structline(0,1)
		\class(3,3) \structline \structline(0,2) \structline(2,2)
		\class(6,2) \structline(3,1)
		\class(7,1)
		\class(7,2) \structline
		\class(7,3) \structline
		\class(7,4) \structline
		\class(8,2) \structline(7,1)
		\class(8,3)

		\foreach \x in {2,...,4} {
				\draw[color=red] (2*\x-3,3) circle [radius=0.16];
				\node[color=red,below] at (2*\x-3.4,3.6) {\theta_{\x}};
			}
	\end{sseqdata}

	\begin{figure}[H]
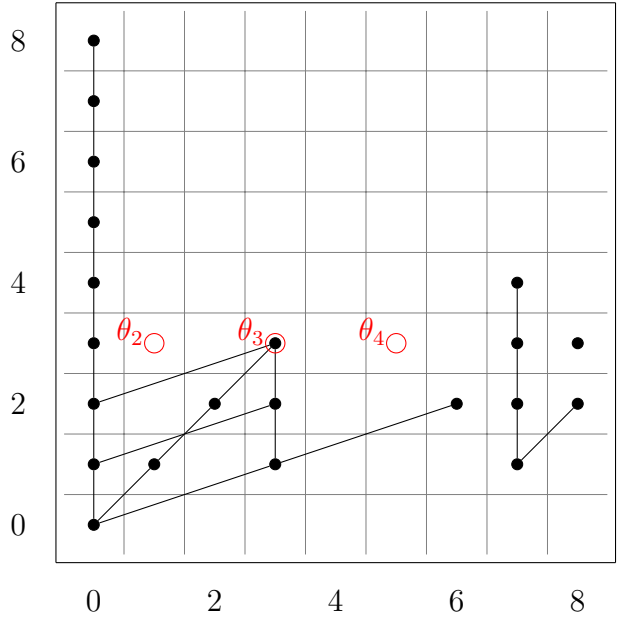

		\centering
		\printpage[ name = ASS, page = 2 ]
		\caption{ Bigraded homotopy groups of \( \nu \mathbb{S} \)
			without \( \tau \)-multiples. Each dot represent a copy of \( \mathbb{F}_2 \).}
		\label{fig:syn1}
	\end{figure}

	We see that the \(\an[4]\)-obstruction lies in a null group if it exists,
	so we just have to show it admits an \(\an[3]\)-structure.
	By \Cref{coolthing},
	since the spectrum \(\sfourmodule\) is \(2\)-local,
	and is the cofiber of the map \(\cdot (x-\tilde{4})\),
	it is enough to show that \(Q_{1}(\cdot(x-\tilde{4}))\) is null.
	We have the following commuting diagram
	\[\begin{tikzcd}
			{[\Sigma^{1,3} \Th(\nu\mathbb{S}),
						\Th(\nu\mathbb{S})]_{\Th(\nu\mathbb{S})}}
			&
			{[\Sigma \Th(\mathbb{S}),\Th(\mathbb{S})]_{\Th(\mathbb{S})}}
			\\
			{[\Sigma^{1,3} \nu \mathbb{S},\Th(\nu\mathbb{S})]_{\syntwo}}
			&
			{[\Sigma \mathbb{S},\Th(\mathbb{S})]_{\Sp},}
			\arrow["{\tauinv}"', hook, from=1-1, to=1-2]
			\arrow["\simeq", from=1-1, to=2-1]
			\arrow["\simeq", from=1-2, to=2-2]
			\arrow["{\tauinv}", hook, from=2-1, to=2-2]
		\end{tikzcd}\]
	where the vertical equivalences are the adjoint morphisms
	of the free forgetful functor between
	\(\rmod_{\Th{\mathbb{S}}}(\Sp)\) and \(\Sp\).
	We can then check if
	\(Q_{1}(\cdot(x-\tilde{4}))\)
	is null,
	by checking \(Q_{1}(x-4)\) is null in \(\Sp\).

	We have the Cartan Formula
	\begin{align*}
		Q_1 ( x - 4 ) = (-4)^{2} Q_1 ( x ) + x^2 Q_1 ( 4 ) + \eta x 4.
	\end{align*}
	We have \( Q_1(4) = 0 \) and \( \eta \) is 2-torsion,
	so the last two terms vanish.
	For \(Q_1(x)\), it can be calculated as
	\begin{align*}
		\mathbb{S}^1 \to
		D_2( \mathbb{S} )
		\xrightarrow{ D_2( x)}
		D_2 (\Th(\mathbb{S}))
		\to
		\Th (\mathbb{S})
	\end{align*}
	where \( D_2(X) = X^{\otimes 2}_{h \Sigma_2} \) is the second extended power.
	The multiplication on
	\( \Th( \mathbb{S}) \simeq \coprod_{n \geq 0} \mathbb{S}  \)
	is given by
	\begin{align*}
		\left( \coprod_{n \geq 0} \mathbb{S} \right)
		\otimes
		\left( \coprod_{m \geq 0} \mathbb{S} \right)
		\simeq
		\coprod_{n,m \geq 0} \mathbb{S} \otimes \mathbb{S}
		\xrightarrow{m}
		\coprod_{n \geq 0} \mathbb{S}.
	\end{align*}
	which sends the \( (n,m) \)'th term to the \( n+m \)'th term with the multiplication map
	of \( \mathbb{S} \).
	From this we get a commuting square
	\[\begin{tikzcd}
			{\mathbb{S}^1} & {D_2( \mathbb{S})} & {D_2( \Th(\mathbb{S})} \\
			& {\mathbb{S}} & {\Th(\mathbb{S}).}
			\arrow[from=1-1, to=1-2]
			\arrow["{D_2(x)}", from=1-2, to=1-3]
			\arrow["m"', from=1-2, to=2-2]
			\arrow["m", from=1-3, to=2-3]
			\arrow["{x^2}", from=2-2, to=2-3]
		\end{tikzcd}\]
	We then see that \( Q_1(x) = Q_1(1)x^2 = 0\), since \( Q_1(1) = 0 \) in \(
	\mathbb{S}\).

\end{proof}
\begin{corollary}\label{main result}
	The synthetic spectrum \(\nu\mathbb{S}/\tilde{4}\)
	and the spectrum \(\mathbb{S}/4\) admit
	\(\mathbb{A}_5\)-mul\-ti\-pli\-ca\-tions.
\end{corollary}
\begin{proof}
	From \Cref{factor},
	we see that \(\nu \mathbb{S}/ \tilde{4}\) admits an \(\an[4]\)-structure,
	and the obstruction for an \(\an[5]\)-structure factors through the map
	\(\nu \mathbb{S} \to \nu \mathbb{S}/\tilde{4}\).
	However from
	\Cref{fig:syn1},
	we see that there are only \(\tilde{4}\)-multiples in \(\pi_{7,3}(\nu
	\mathbb{S})\) so the obstruction must vanish. By inverting \(\tau\),
	we also get an \(\an[5]\)-structure on \(\mathbb{S}/4\).
\end{proof}

\appendix
\section{Localising
  Stable Categories at Primes}\label{bousfield}
We recall localization of presentable \(\infty\)-categories,
and the properties we need.
\begin{definition}
	Given \(E \in \Sp\),
	a map \( f: X \to Y \) in \( \Sp \) is an \textit{\(E\)-equivalence},
	if \( E \otimes f: E \otimes X \to E \otimes Y \)
	is an equivalence of spectra.
	A spectrum \( Z \) is \textit{\( E \)-local} if for every \( E
	\)-equivalence \( f: X \to Y \),
	\begin{align*}
		\map ( Y, Z)
		\xrightarrow{f_*}
		\map ( X , Z )
	\end{align*}
	is an equivalence.
	An \textit{\( E \)-localization} of \( X \) is an
	\( E \)-local spectrum \( L_E X \),
	with an \( E \)-equivalence \( X \to L_E(X) \).
\end{definition}
An \( E \)-localization always exists and is unique.
\begin{example}
	Localization with respect to the Moore spectra \( E = S \mathbb{Z}_{(p)} \)
	is called \( p \)-localization.
	The spectrum \( \mathbb{S} / 4 \) is \( 2 \)-local.
\end{example}
We denote the full \( \infty \)-subcategory of \( \Sp \)
spanned by \( p \)-local spectra by \( \Sp_{(p)} \).
By \cite[Proposition 2.2.1.9]{HA},
the \( p \)-localization functor is symmetric monoidal,
and the inclusion is lax symmetric monoidal.
Since lax symmetric monoidal functors induce
functors on algebra categories,
we get a functor
\begin{align*}
	\alg_{\an}( \Sp_{(p)} ) \to \alg_{\an}( \Sp).
\end{align*}
Therefore to prove that \( \mathbb{S} / 4 \) admits an \( \an \)-algebra structure,
it is enough to prove it is an \( \an \)-algebra in \( \Sp_{(p)} \),
as it is the cofiber of \(4: \mathbb{S}_{(2)} \to \mathbb{S}_{(2)} \),
where \( \mathbb{S}_{(2)} \) is the 2-local sphere.
\begin{remark}

	We have that the pair  \( ( \Sp_{(2)}, \mathbb{S}_{(2)})   \) is idempotent in \(
	\prl \) in the sense of \cite[page 720]{HA},
	so the forgetful functor \( \mathrm{Mod}_{\Sp_{(2)}}( \prl) \to \prl \)
	determines
	a fully faithful embedding,
	whose left adjoint is given by the tensoring with
	\(\Sp_{(2)}\) in \(\prl\).
\end{remark}
\begin{definition}
	We define the \( \infty \)-category of presentable stable 2-local categories
	\( \prl_{(2)}\) as \( \mathrm{Mod}_{\Sp_{(2)}}( \prl) \).
\end{definition}
\begin{lemma}
	We have an equivalence in \( \prl_{(2)} \)
	\begin{align*}
		\fun ( \mathcal{C} ,\Sp ) \otimes \Sp_{(2)}
		\simeq
		\fun ( \mathcal{C}, \Sp_{(2)}).
	\end{align*}

\end{lemma}
\begin{proof}
	Using \cite[Proposition 4.8.1.17]{HA}, we have
	\begin{align*}
		\fun ( \mathcal{C} ,\Sp ) \otimes \Sp_{(2)}
		\simeq
		\rfun( \fun(\mathcal{C},\Sp)^{\op} , \Sp_{(2)})
		\\
		\simeq
		\lfun( \fun(\mathcal{C},\Sp) , \Sp_{(2)}^{\op})^{\op}
		\\
		\simeq
		\fun( \mathcal{C}^{\op}, \Sp_{(2)}^{\op})^{\op}
		\\
		\simeq
		\fun( \mathcal{C}, \Sp_{(2)})
	\end{align*}
\end{proof}
In particular, we have
\( \Sp^{\gr} \otimes \Sp_{(2)} \simeq \Sp_{(2)}^{\gr} \).


\printbibliography

\end{document}